\documentclass{amsart}
\usepackage{bbm}
\usepackage{amsmath,amsthm,amsfonts,amssymb,amscd,mathrsfs}
\usepackage{graphicx}
 \usepackage[all]{xy}
 
\usepackage{xcolor}
\usepackage{boxhandler}
\usepackage{caption}
\usepackage{lipsum}
\newtheorem*{Mheo}{Main Theorem}
\newtheorem{theo}{Theorem}[section]
\newtheorem{prop}{Proposition}[section]
\newtheorem{lemma}{Lemma}[section]

\newtheorem{rem}{Remark}[section]

\begin{document}


\title{Embedding asymptotically expansive systems}

\author{DAVID BURGUET}

\address{LPMA - CNRS UMR 7599 \\ 
\\
Universite Paris 6 \\
\\
 75252 Paris Cedex 05 FRANCE\\
\\
david.burguet@upmc.fr}


\begin{abstract}A topological dynamical system is said asymptotically expansive when entropy and periodic points grow subexponentially at arbitrarily small scales.   We prove a Krieger like embedding theorem for asymptotically expansive systems with the small boundary property. We show that such a system $(X,T)$ embeds in the  $K$-full shift if $ h_{top}(T)<\log K$ and $\sharp Per_n(X,T)\leq K^n$ for any integer $n$. The embedding is  in general not continuous  (unless the system is expansive and $X$ is zero-dimensional) but the induced map on the set of invariant measures is a topological embedding. It is  shown that this property implies asymptotical expansiveness. We prove also that the inverse of the embedding map may be continuously extended to a faithful principal symbolic extension.  
\end{abstract}

\maketitle
\section{Introduction}
Symbolic dynamics play since the pioneer work of Hadamard \cite{Had} a crucial role in the theory of dynamical systems. 
Here we investigate the problem of embedding a dynamical system in a shift with a finite alphabet.

For measure preserving ergodic systems the celebrated Krieger generator theorem gives a complete answer in terms of measure theoretical entropy : 

\begin{theo}\label{kkk} (Theorem 2.1 and 4.3 in \cite{Kr})
Let $T$ be an ergodic automorphism of a standard probability space $(X,\mathcal{U}, \mu)$ then $T$ embeds measure theoretically in the  shift with $K$ letters, i.e. there exists a measurable injective   map $\psi:X\rightarrow \{1,...,K\}^\mathbb{Z}$ with $\sigma \circ \psi=\psi\circ T$ if and only if:  \\

either $ h_{\mu}(T)<\log K$,

or $h_\mu(T)=\log K$ and $T$ is Bernoulli. 
\end{theo}

For topological dynamical systems (i.e. continuous maps on a compact metrizable space) this question was also solved by Krieger. The obstructions are now  of three kinds : topological ($X$ is  zero-dimensional), set theoretical (the number of $n$-periodic points of $(X,T)$ is less than  or equal to the number of $n$-periodic points of the shift) and dynamical ($(X,T)$ is expansive and its topological entropy is less than the entropy of the shift).  For any integer $n$ we let $Per_n(X,T)$ be the set of periodic points of $(X,T)$ with least period equal to $n$.

\begin{theo}\label{kkkk}(Theorem 3 in \cite{Kri})
Let $(X,T)$ be a topological dynamical system, then $(X,T)$ embeds topologically in the shift $\sigma$ with $K$ letters, i.e. there exists a continuous injective map $\psi:X\rightarrow \{1,...,K\}^\mathbb{Z}$ with $\sigma \circ \psi=\psi\circ T$ if and only if: \\

either we have \begin{itemize}
\item $X$ is zero-dimensional, 
\item $\frac{1}{n}\log\sharp Per_n(X,T)\leq \log K$ for any integer $n$,
\item $h_{top}(T)<\log K$,
\item $T$ is expansive,
\end{itemize}

or $(X,T)$ is topologically conjugated to $( \{1,...,K\}^\mathbb{Z},\sigma)$ (and thus $h_{top}(T)=\log K$).
\end{theo}

Krieger only deals with the case $h_{top}(T)<\log K$, however the case of equality may be proved as in Theorem \ref{kkk} (See the Appendix).\\

To embed topologically more general systems  (in particular of arbitrarily topological dimension) into shift spaces one has to consider the shift with alphabet in $[0,1]^k$, $k\in \mathbb{N}\cup\{\mathbb{N}\}$. The existence of an embedding in the shift over $[0,1]^k$ with finite $k$ is one important question in the theory of mean dimension, see e.g. \cite{Lin} and references therein (for $k=\mathbb{N}$ it always exists by considering an embedding of $X$ into the Hilbert cube $[0,1]^\mathbb{N}$). If we want still work with the shift with a finite alphabet we have to consider a larger class of embedding maps.  
Recently Hochman gives a  Borel version of Krieger theorem for Borel dynamical systems, where the embedding is  a Borel map.  In this setting, after forgetting periodic points,   the only constraint is the supremum of the entropy of  Borel ergodic invariant probability measures:

\begin{theo}(Theorem 1.5 in \cite{Hoch})
Let $T$ be a measurable automorphism of a standard Borel  space $(X, \mathcal{U})$ then $(X,T)$ embeds almost Borel  in the shift with $K$ letters, i.e.
there exists a Borel subset $E$ of $X$ of full measure with respect to any ergodic $T$-invariant aperiodic measure and a Borel injective map $\psi:E\rightarrow \{1,...,K\}^\mathbb{Z}$ with $\sigma \circ \psi=\psi\circ T$, if and only if: \\

either we have  $ h_{\mu}(T)<\log K$  for  any Borel ergodic $T$-invariant probability measure $\mu$,
  
 or $(X,T)$ admits a unique  measure $\mu$ of maximal entropy, $\mu$ is Bernoulli and $h_\mu(T)=\log K$.

\end{theo}

In the present paper we give a version of Krieger theorem for asymptotically expansive topological dynamical system with the small boundary property. A dynamical system is asymptotically expansive when entropy and periodic points grow subexponentially at arbitrarily small scales.  The small boundary property  is satisfied by many aperiodic systems, i.e. dynamical systems without periodic points. In particular any aperiodic system in finite dimension satisfies this property. Precise definitions and further known facts related to these two properties  are given in Section 2. 

 Our embedding is not continuous but it is in some sense precised later a limit of essentially continuous functions. However the map induced by the  embedding on the set of invariant measures is  continuous. Conversely we prove that a topological dynamical system  embedding in such a way in a shift with a  finite alphabet is asymptotically expansive. \\

Finally we relate for asymptotically expansive systems the present Krieger embedding theorem with the theory of symbolic extensions. More precisely the inverse of the Krieger embedding defines a principal faithful symbolic extension.   In this way we build  a bridge between the theory of symbolic extensions developed by M. Boyle and T. Downarowicz \cite{BFF}\cite{BD}\cite{Dowb} and Krieger  embedding like problems, for asymptotically expansive systems. In  these previous works   the  symbolic extension is built as the intersection of decreasing subshifts whereas we build here the  extension as the closure of increasing subshifts. \\

In Section 2 we introduce the main notions. We then state our main theorem.  Section 4 is devoted to the proof of our embedding Theorem for zero-dimensional asymptotically expansive systems. In Section 5 we deal with the general case by reduction to the zero-dimensional one. Finally we will see that  systems embeddable in our sense are asymptotically expansive. In this way we get a new characterization of asymptotical expansiveness.

\section{Background}
We will always consider topological dynamical systems $(X,T)$, i.e. $X$ is a compact metrizable space and $T:X\rightarrow X$ is an invertible continuous map. We also always assume that $(X,T)$ has finite topological entropy. We denote by $\mathcal{M}(X,T)$ the set of $T$-invariant Borel probability measures and we endow this set with the weak $*$-topology. 

\subsection{Small boundary property and essential partitions}\label{SBP}
We recall here  some facts about the small boundary property. A Borel subset of $X$ is called a \textbf{null} (resp. \textbf{full}) set if it has null (resp. full) measure for any $T$-invariant ergodic Borel probability measure.  

A subset of $X$ is said to have a \textbf{small boundary} when its boundary is a null set. A partition of $X$ is called \textbf{essential} when any of its element  have a  small boundary. For a zero-dimensional topological dynamical system $(Y,S)$, a Borel map $\psi:X\rightarrow Y$ is said to be \textbf{essentially continuous} if there exists a basis of clopen sets $\mathcal{B}$ of $Y$ such that for any $B\in \mathcal{B}$ the set $\psi^{-1}(B)$ has a small boundary. Observe that it easily implies that the map induced by $\psi$ from  $\mathcal{M}(X,T)$ to $\mathcal{M}(Y,S)$  is continuous.  We also say finally that $(X,T)$ has the \textbf{small boundary property} if $X$ admits finite essential partitions  with arbitrarily small diameter.

 This property has been investigated  in \cite{Lin1}, \cite{Lin99}, \cite{Gut} and  used in the theory of symbolic extensions and entropy structures developed in \cite{BD},\cite{Dowb}. In particular any dynamical system of finite  topological entropy with an aperiodic minimal factor (Theorem 6.2 in \cite{Lin99}) or any finite dimensional aperiodic system (Theorem 3.3 of \cite{Lin1}) satisfies the small boundary property. In fact in  \cite{Lin1}  it is proven that  any finite dimensional system admits a basis of neighborhoods whose boundaries have  zero measure for any aperiodic invariant measures. When $Per(X,T)$ is a zero-dimensional subset of $X$ one may easily arrange the construction of \cite{Lin1}  to ensure that any element of  the basis has a small boundary. Thus any finite dimensional system with a zero-dimensional set of periodic points has the small boundary property.  This result may also follow  from Lemma 3.7 of \cite{Kuz}.
  We refer to \cite{Gut} for some others dynamical properties implying the small boundary property.

\subsection{Asymptotic $h$-expansiveness}

Given two finite open covers $\mathcal{U}$ and $\mathcal{V}$ of $X$, we define the topological conditional 
entropy $h(\mathcal{V}|\mathcal{U})$ of $\mathcal{V}$ given $\mathcal{U}$ as 

$$h(\mathcal{V}|\mathcal{U}):=\lim_n\frac{1}{n} \sup_{U^n\in \mathcal{U}^n}\log\min \sharp\{ \mathcal{F}_n\subset \mathcal{V}^n, \  U^n \subset \bigcup_{V^n\in \mathcal{F}_n}V^n\}.$$

A map is said to be \textbf{asymptotically $h$-expansive} when for any decreasing sequence of  open covers $(\mathcal{U}_k)_k$ whose diameter goes to zero we have

$$ \lim_k\sup_{\mathcal{V}}h(\mathcal{V}|\mathcal{U}_k)=0,$$

or equivalently  when we have

$$ \inf_{\mathcal{U}}\sup_{\mathcal{V}}h(\mathcal{V}|\mathcal{U})=0.$$

We refer to \cite{Dowb} for basic properties of the topological conditional entropy and different characterizations of asymptotically $h$-expansiveness.

These notions were introduced by Misiurewicz in the seventies. One important consequence of the asymptotically $h$-expansiveness property is  the upper semicontinuity of the measure theoretical entropy function and thus the existence of a measure of maximal entropy. A large class of dynamical systems satisfies this property, e.g. continuous piecewise monotone interval maps \cite{Mis}, endomorphisms on compact groups, $C^\infty$ maps on compact manifolds \cite{Buzz},...  \\

A topological extension  $\pi:(Y,S)\rightarrow (X,T)$ is called \textbf{principal} when it preserves the entropy of invariant measures, i.e. $h_\nu(S)=h_\mu(T)$ for any $T$-invariant measure $\mu$ and for any $S$-invariant measure $\nu$ projecting on $\mu$. Ledrappier \cite{Led} proved that if $\pi$ is a principal extension then $T$ is asymptotically $h$-expansive if and only if the same holds for $S$.

\subsection{Asymptotic $per$-expansiveness}
Similarly we introduce now the new notion of asymptotical $per$-expansiveness. With the previous notations  we let 

$$per(T|\mathcal{U}):=\limsup_n\frac{1}{n} \sup_{U^n\in \mathcal{U}^n}\log\sharp  Per_n(X,T)\cap U^n$$

 and we say $(X,T)$ is  \textbf{asymptotically $per$-expansive} when we have
$$\lim_k per(T|\mathcal{U}_k)=0,$$

or equivalently  when we have
$$ \inf_{\mathcal{U}}per(T|\mathcal{U})=0.$$

We recall that a topological system $(X,T)$ is expansive when there exists a finite open cover $\mathcal{U}$ 
with the following property:  for any $x\neq y\in X$ there is  $k\in \mathbb{Z}$ such that $T^kx$ and $T^ky$ do not lie in the same element of $\mathcal{U}$. 

Obviously aperiodic systems  and expansive systems are asymptotically $per$-expansive. We say that a  topological dynamical system is  \textbf{asymptotically expansive} when it is both asymptotically $h$- and $per$-expansive. Among asymptotically expansive systems which are neither expansive nor apeiodic we can cite \cite{Bur16} toral quasi-hyperbolic automorphisms, piecewise expanding map of the interval or also generic piecewise affine surface homeomorphisms. 
The author   recently proved  \cite{Bur16} that any $C^\infty$ surface diffeomorphism, whose  invariant measures have one negative and one positive Lyapunov exponents uniformly away from zero, is asymptotically per-expansive (observe that any periodic point is then hyperbolic, in particular the set of periodic points is zero-dimensional and the system has the small boundary property).  In particular  $C^\infty$ Henon-like diffeomorphisms satisfy this property \cite{Ber} (see also \cite{Tak}).\\

  A topological extension is said to be \textbf{faithful} if the induced map between the sets of invariant probability measures is an homeomorphism. As any  system has a principal faithful aperiodic extension (even zero-dimensional, see \cite{DHH}), the factor of a  asymptotically $per$-expansive map is not necessarily asymptotically $per$-expansive, even when the entropy of measures is preserved and the extension is faithful. Thus asymptotically $per$-expansiveness is not preserved under principal faithful extensions. 
  
  A topological extension  $\pi:(Y,S)\rightarrow (X,T)$ is said to be \textbf{strongly faithful} if it is faithful and if it preserves periodic points, i.e. for any integer $n$ we have $\pi(Per_n(Y,S))=Per_n(X,T)$.  In Subsection \ref{dd}  we will show that a dynamical system with a principal  strongly faithful asymptotically $per$-expansive extension is also  asymptotically $per$-expansive.\\

\subsection{Symbolic extensions}

A \textbf{symbolic extension} of $(X,T)$ is a topological extension by a subshift over a finite alphabet. 
 The question of the existence of (principal, faithful) symbolic extension has led to a deep theory of entropy (we refer to \cite{Dowb} for an introduction to the topic). One first positive result appeared in this area is the existence of principal  symbolic extensions for asymptotically $h$-expansive systems \cite{BFF}, \cite{Down}. More recently Serafin \cite{Ser} proved that such an extension could be chosen to be faithful when $(X,T)$ is aperiodic\footnote{ This was extended by   Downarowicz and Huczek to any asymptotically $h$-expansive systems \cite{DHH}}. Here we give a new proof of these results that we relate with our Krieger like embedding  theorem (Main Theorem below) : the symbolic extension is just given by the inverse of the Krieger embedding.

\section{Krieger embedding  for asymptotically expansive systems}

We state now our main result. For two topological dynamical systems  $(X,T)$ and $(Y,S)$, a  map $\phi:X\rightarrow Y$ is called \textbf{equivariant}  when it semi-conjugates $T$ with $S$, i.e. $\phi \circ T=S\circ \phi$. Moreover we say $\phi:X\rightarrow Y$ is \textbf{$\epsilon$-injective} for some $\epsilon>0$, if there exists $\delta>0$, such that for any set $Z\subset Y$ with diameter less than $\delta$ the  preimage  $ \phi^{-1}(Z)$  has diameter less than $\epsilon$. Finally we will denote by $\phi^*:\mathcal{M}(X,T)\rightarrow \mathcal{M}(Y,S)$ the map induced by $\phi$ between the sets of probability invariant measures.

\begin{Mheo}\label{main}\nonumber
Let $(X,T)$ be a topological dynamical system  with the following properties : 

\begin{itemize}
\item $(X,T)$ has the small boundary property,
\item $\frac{1}{n}\log\sharp Per_n(X,T)\leq \log K$ for any integer $n$,
\item $h_{top}(T)<\log K$,
\item $(X,T)$ is asymptotically expansive.\\
\end{itemize}

 Then there exists an  equivariant injective Borel map $\psi:X\rightarrow \{1,...,K\}^\mathbb{Z}$ such that :
\begin{itemize}
\item   $\psi$ is a pointwise limit of a sequence of equivariant essentially continuous  $\epsilon_k$-injective maps $\psi_k$, where $(\epsilon_k)_k$ is going to zero,
\item the induced map $\psi^*$ is the uniform limit of $(\psi_k^*)_k
$ on $\mathcal{M}(X,T)$,  in particular $\psi^*$  is a topological embedding,
\item $\psi^{-1}$ is uniformly continuous on $\psi(X)$ and   the continuous extension $\pi$ of $\psi^{-1}$ on the closure $Y=\overline{\psi(X)}$ of $\psi(X)$ is a principal strongly faithful symbolic extension of $(X,T)$ with $\psi^*=(\pi^*)^{-1}$.
\end{itemize}
\end{Mheo}

As previously discussed in Subsection \ref{SBP} the above theorem applies to any asymptotically expansive finite dimensional system  with finite topological entropy and finite exponential growth of periodic points, since in this case the set of periodic points is zero-dimensional. Observe also that the asymptotic $per$-expansiveness and the inequality $h_{top}(T)<\log K$ implies that the exponential growth of periodic point $per(T):=\limsup_n\frac{\log \sharp Per_n(X,T)}{n}$ also satisfies 
$per(T)\leq h_{top}(T)< \log K$.\\

The dynamical consequences of our statement characterize the asymptotic expansiveness. More precisely we have as 
 a converse of our Main Theorem :

\begin{prop}\label{embb}
Let $(X,T)$ be a topological dynamical system.  Assume one of the two following properties : 

\begin{enumerate}
\item there exists a Borel equivariant injective map $\psi:X\rightarrow  \{1,...,K\}^\mathbb{Z}$  such that $\psi^*$ is a topological embedding,
\item there exists a strongly faithful principal extension  $\pi: (Y,\sigma)\rightarrow (X,T)$ with $Y$ a closed $\sigma$-invariant  subset of  $\{1,...,K\}^\mathbb{Z}$.\\
\end{enumerate}

Then $(X,T)$ satisfies the following properties :
 \begin{itemize}
\item $\frac{1}{n}\log\sharp Per_n(X,T)\leq  \log K$ for any integer $n$, \\
\item $h_{top} (T)\leq \log K$,\\
\item  $(X,T)$ is asymptotically  expansive. 
\end{itemize}
\end{prop}

Together with the Main Theorem we obtain in  particular that when $h_{top}(T)<\log K$ and $X$ has the small boundary property then the assumptions (1) and (2) in Proposition \ref{embb} are equivalent.\\

In contrast with Theorem \ref{kkk} and \ref{kkkk}, for topological systems $(X,T)$ with $h_{top}(T)=\log K$, the existence of an embedding on the set of invariant measures of $(X,T)$ and the $K$-full shift does not seems to create other constraints on $(X,T)$: in particular such an embedding may not be onto as we can see  in the following example. 

 There exists by Jewett-Krieger ergodic theorem for any integer $K$ a uniquely ergodic  dynamical system with topological entropy equal to $\log K$ (such a system is in particular aperiodic). 
It easily follows from the tail variational principle \cite{burg} that uniquely ergodic systems are asymptotically $h$-expansive. When the unique ergodic measure is Bernoulli then Theorem \ref{kkk} gives a Borel equivariant map inducing a non surjective  embedding in the set of invariant measures of the $K$-full shift.

However the case of equality in the conditions on periodic points creates some "rigidity".  Indeed, if we assume  $\sharp Per_n(X,T)=\sharp Per_n(\{1,...,K\}^\mathbb{Z},\sigma)$ for any integer $n$, then for any Borel equivariant map $\psi:X\rightarrow  \{1,...,K\}^\mathbb{Z}$  such that $\psi^*$ is a topological embedding, the map $\psi^*$ is in fact an homeomorphism  because periodic measures are dense in the set of invariant measures of the full shift.\\

As the previously mentioned embedding theorems our main result brings some new contributions in the classification of dynamical systems. In \cite{dowi} T. Downarowicz investigates the characterization of \textit{assignments} arising from  topological systems. An assignment is a function $\psi$ defined on an abstract metrizable
Choquet simplex $K$, whose values are measure-theoretic dynamical systems,
i.e., for $p\in K$, $\psi(p)$ has the form $(X_p, \mathcal{B}_p, \mu_p, T_p)$. Two assignments, $\psi$ on
a simplex $K$ and $\psi'$ on a simplex $K'$, are said to be equivalent if there exists
an affine homeomorphism of Choquet simplexes $\pi: K \rightarrow  K'$ such that for every
$p\in K$ the systems $\psi(p)$ and  $\psi'(p')$, where $p' = \pi(p)$, are isomorphic. A
topological dynamical system $(X, T)$ determines a natural assignment on the
simplex $\mathcal{M}(X,T)$ by the rule $\mu \mapsto  (X, \mathcal{B}, \mu, T)$ with $\mathcal{B}$ being the usual Borel $\sigma$-algebra of $X$. As a consequence of our Main Theorem, subshifts and asymptotically expansive systems have the same assignments, i.e. the natural assignment of an asymptotically expansive system is equivalent to the  natural assignment of some subshift.

\section{The case of zero-dimensional systems}
We first consider zero-dimensional dynamical systems. Then in Section \ref{red} we will deal with general systems with the small boundary property by reduction to the zero-dimensional  case.

We will prove the following strong version (continuity replaces essential continuity) of our Main theorem for zero-dimensional systems :

\begin{prop}\label{zer}
Let $(X,T)$ be a zero-dimensional topological dynamical system with the following properties : 

\begin{itemize}
\item $\sharp Per_n(X,T)\leq \sharp Per_n(\{1,...,K\}^\mathbb{Z},\sigma)$ for any integer $n$,
\item $h_{top}(T)<\log K$,
\item $(X,T)$ is asymptotically expansive.\\
\end{itemize}

 Then there exists an  equivariant injective Borel map $\psi:X\rightarrow \{1,...,K\}^\mathbb{Z}$ such that :
\begin{itemize}
\item   $\psi$ is a pointwise limit of a sequence of equivariant  continuous  $\epsilon_k$-injective maps $\psi_k$, where $(\epsilon_k)_k$ is going to zero,
\item the induced map $\psi^*$ is the uniform limit of $(\psi_k^*)_k
$  on $\mathcal{M}(X,T)$, in particular $\psi^*$  is a topological embedding,
\item $\psi^{-1}$ is uniformly continuous on $\psi(X)$ and  the continuous extension $\pi$ of $\psi^{-1}$ on the closure $Y=\overline{\psi(X)}$ of $\psi(X)$ is a principal strongly faithful symbolic extension of $(X,T)$ with $\psi^*=(\pi^*)^{-1}$.
\end{itemize}
\end{prop}

In the next Subsections 4.1, 4.2 and 4.3 we develop some tools used later in the proof of Proposition \ref{zer}.  

\subsection{Asymptotic expansivenness for zero-dimensional systems}
Let $(X,T)$ be a zero-dimensional system. In the definition of asymptotic $h$-expansiveness given in Subsection 2.2 we may choose the open covers $\mathcal{U}$ and $\mathcal{V}$ to be finite clopen partitions with $\mathcal{V}>\mathcal{U}$, i.e. any element $V$ of $\mathcal{V}$ is contained in a  element $U$ of $\mathcal{U}$. The topological conditional entropy 
$h(\mathcal{V}|\mathcal{U})$ may be then rewritten as follows 

$$h(\mathcal{V}|\mathcal{U}):=\lim_n\frac{1}{n} \sup_{U^n\in \mathcal{U}^n}\log\sharp\{ V^n\in \mathcal{V}^n, \ V^n\subset U^n\}.$$

The following lemma used later in the proof of Proposition  \ref{zer} follows directly from the definition of asymptotical  expansiveness independently from the above expression of the topological conditional entropy.

\begin{lemma}\label{hex}Let $(X,T)$ be an asymptotically expansive  zero-dimensional system. For all $\alpha>0$  there exists a decreasing sequence  $(\mathcal{V}_k)_k$ of clopen partitions whose diameter goes to zero, such that  for all $k\geq 1$:

\begin{itemize}
\item $h(\mathcal{V}_{k+1}|\mathcal{V}_k)<\alpha/2^k$,
\item $per(T|\mathcal{V}_k)<\alpha/2^k$.
\end{itemize}
\end{lemma}

\subsection{The  marker property for zero-dimensional systems} One key tool in our construction is the following  "marker property". A similar approach is used to build symbolic extensions in \cite{BD} where the product with an odometer is used to mark the shift space in the same way.

\begin{lemma}(Lemma 7.5.4 in \cite{Dowb})\label{mar}
Let $(X,T)$ be a zero-dimensional  dynamical system. Then  for every positive integer $n$ and for every   $\epsilon$  there exists a  clopen set $U$ such that  :
\begin{itemize}
\item  $T^iU$ are pairwise disjoint for $i=0,...,n-1$,
\item  $\bigcup_{|i|<n}T^{i}U= X\setminus Per_{n}^{\epsilon}$, where  $ Per_{n}^{\epsilon}$ denotes a clopen  $\epsilon$-neighborhood of the set $\bigcup_{m\leq n}Per_{m}(X,T)$ of periodic points with least period less than or equal to $n$. 
\end{itemize}
\end{lemma}

\subsection{Some tools on shift spaces}
We consider in this section a finite set $\Lambda$ and a compact metrizable space $X$. We let $P_0$ be the zero coordinate partition of $\Lambda^\mathbb{Z}$.

\subsubsection{Decreasing metrics  on  $\Lambda^\mathbb{Z}$}

 The product space $\Lambda^\mathbb{Z}$ will be endowed with different metrics. We first let $d$ be a metric inducing the usual product topology (we consider the discrete topology on $\Lambda$), for example the Cantor metric defined as  
 $$\forall x,y\in \Lambda^\mathbb{Z}, \ d(x,y):=2^{-k} \text{ where } k=\min\{i\geq 0, \ x_{-i}\neq y_{-i} \text{ or } x_i\neq y_i\}.$$
Also  for any integer $N$ we consider the metric $d_N$ defined as $$\forall x,y\in \Lambda^\mathbb{Z}, \ d_N(x,y):=\sup_{n\geq N}\left\{\frac{1}{2n+1}\sharp \{i, \  |i| \leq n , \ x_i\neq y_i\right\}.$$ For any $N$ the topology given by $d_N$ is stronger than the product topology.   Observe now that the sequence of metrics $(d_N)_N$ is nonincreasing, i.e. $(d_N(x,y))_N$ is nonincreasing in $N$ for any $x,y\in \Lambda^\mathbb{Z}$. We let $d_\infty$ be the (shift-invariant) pseudometric on $\Lambda^\mathbb{Z}$ given by the limit of $(d_N)_N$ 

$$d_\infty(x,y):=\lim_Nd_N(x,y).$$

The pseudometric $d_\infty$ is called the Besicovitch pseudometric on the shift space. Topological properties of the induced quotient  metric space and dynamical properties of cellular automata on this space were studied in \cite{Bl}. In particular, although we do not used here  directly, this quotient metric space is known to be complete. 

\subsubsection{Convergence of  functions taking values in $\Lambda^\mathbb{Z}$}

We consider now maps from $X$  to $\Lambda^\mathbb{Z}$  and we define different kinds of convergence. A sequence of such maps $(\psi_k)_k$ is said to converge to $\psi$ :

\begin{itemize}
\item \textbf{pointwisely with respect to $d$}, when for all $x\in X$, 
$$ d(\psi_kx,\psi x)\xrightarrow{k\rightarrow +\infty}0,$$
\item \textbf{uniformly  with respect to  $d_\infty$}, when 
$$ \sup_{x\in X}d_\infty(\psi_kx,\psi x)\xrightarrow{k\rightarrow +\infty}0,$$
\item \textbf{uniformly with respect to  $(d_N)_N$}, when 
$$ \sup_{x\in X}d_N(\psi_k x,\psi x)\xrightarrow{N,k\rightarrow +\infty}0.$$
\end{itemize}

Observe that if $(\psi_k)_k$ is converging to $\psi$ uniformly with respect to the decreasing sequence of pseudometrics $(d_N)_N$,  then it is also converging  to $\psi$ uniformly  with respect to its limit $d_{\infty}$.\\

We let $\mathcal{M}(\Lambda^\mathbb{Z}, \sigma)$ be the set of probability Borel measures endowed with the following  metric $d_*$ inducing the weak $*$-topology : $\forall \mu, \nu \in \mathcal{M}(\Lambda^\mathbb{Z},\sigma)$,
$$d_*(\mu,\nu)=\sum_n\frac{|\mu(A_n)-\nu(A_n)|}{2^n},$$
where $(A_n)_n$ is a given enumeration of $\bigcup_{N\in \mathbb{N}}P_0^N$.
 We also consider the space $\mathcal{K}(\Lambda^\mathbb{Z}, \sigma)$  of compact subsets of $\mathcal{M}(\Lambda^\mathbb{Z},\sigma)$ endowed with the Hausdorff metric $d_H$ associated to $d_*$. The set $\mathcal{M}(\Lambda^\mathbb{Z},\sigma)$ being compact it is well known that $\mathcal{K}(\Lambda^\mathbb{Z}, \sigma)$ is also compact.  For all $x\in \Lambda^\mathbb{Z}$  we let $\phi(x)\in \mathcal{K}(\Lambda^\mathbb{Z},\sigma)$ be the set of limits of the empirical measures, i.e. the accumulation points of the sequence $(\frac{1}{n}\sum_{0\leq k<n}\delta_{\sigma^k x})_n$ for the weak $*$-topology.

\begin{lemma}\label{emp}
Let $\psi:X\rightarrow \Lambda^\mathbb{Z}$  be a uniform limit with respect to $d_\infty$ of $(\psi_k)_k$. Then 
$(\phi\circ \psi_k)_k$ converge uniformly to $\phi\circ \psi$ with respect to $d_H$.
\end{lemma}

When $X$ is $\Lambda^\mathbb{Z}$  and  $x,y$ are two points in $\Lambda^\mathbb{Z}$ with $d_\infty(x,y)=0$, then by taking $\psi$ constant equal to $x$ and $\psi_k$ constant equal to $y$ for any $k$, we get then  $\phi(x)=\phi(y)$.

\begin{proof}
It is enough to prove that for any $A\in \bigcup_{N\in \mathbb{N}}P_0^N$ we have 

$$\lim_k\sup_{x\in X}\limsup_n  \frac{1}{n}\left|\sum_{0\leq l<n}\left(\delta_{\sigma^l \circ \psi_k (x)}(A)-\delta_{\sigma^l \circ \psi (x)}(A)\right)\right|=0.$$
  
 Fix $N\in \mathbb{N}$, $A\in P_0^N$, $x\in X$ and $k\in \mathbb{N}$. Then 
\begin{eqnarray*}
\limsup_n \frac{1}{n}\left|\sum_{0\leq l<n}\left(\delta_{\sigma^l \circ \psi_k (x)}(A)-\delta_{\sigma^l \circ \psi (x)}(A)\right)\right| &\leq &  \limsup_n \frac{1}{n}\sum_{0\leq l<n}\left|\delta_{\sigma^l \circ \psi_k (x)}(A)-\delta_{\sigma^l \circ \psi (x)}(A)\right|,\\
&\leq & \limsup_n  \frac{N}{n}\sharp \{0\leq l\leq n+N, \  (\psi_k(x))_l\neq (\psi(x))_l\},\\
& \leq & N d_\infty(\psi_k(x),\psi(x)).
\end{eqnarray*}

By uniform convergence of $(\psi_k)_k$ to $\psi$ with respect to $d_\infty$, this last term goes to zero uniformly in $x$ when $k$ goes to infinity. This concludes  the proof the lemma.
\end{proof}

For a subset $Y$ of $\Lambda^\mathbb{Z}$ we let $\overline{Y}$ be the closure of $Y$ for the product topology and we let $Y^\infty$ be the $d_\infty$-saturated set of $Y$, i.e. $Y^\infty=\{x\in \Lambda^\mathbb{Z}, \ \exists y\in Y, \ d_\infty(x,y)=0\}$. 

\begin{lemma}\label{sat}
Let $\psi:X\rightarrow \Lambda^\mathbb{Z}$  be both a pointwise limit with respect to $d$ and a uniform limit with respect to $(d_N)_N$  of $(\psi_k)_k$. Assume moreover the maps $(\psi_k)_k$ are 
 continuous (for the  product topology on $\Lambda^\mathbb{Z}$), then 
$$\overline{\psi(X)}\subset \psi(X)^\infty.$$
\end{lemma}

\begin{proof}
Let $(y_n=\psi(x_n))_n$ be a sequence  converging with respect to $d$ in $\psi(X)$ to say $y$. We consider a sequence of continuous maps $(\psi_k)_k$ converging to $\psi$ uniformly with respect to $(d_N)_N$. By definition for all $\epsilon>0$ there exist integers  $K$ and $M$ such that for all $k\geq K$ and for all $N\geq M$ we have  for all integers $n$ that $d_N(\psi(x_n),\psi_k(x_n))\leq \epsilon$. Observe the function $(x,y)\in (X,d)^2\mapsto d_N(x,y)$ is lower semicontinuous as a supremum of continuous functions. Up to extracting a subsequence we may assume by compacity of $X$ that $(x_n)_n$ is converging to $x\in X$. Therefore by taking the limit in $n$ in the previous inequality we obtain by continuity of $\psi_k$  for all $k\geq K$ and for all $N\geq M$  
\begin{eqnarray*}
d_N(y,\psi_k(x))&\leq &\epsilon.
\end{eqnarray*}

and then by pointwise convergence of $(\psi_k)_k$ to $\psi$ in $(\Lambda^\mathbb{Z},d)$ we have for all $N\geq M$
$$d_N(y, \psi(x))\leq  \epsilon.$$

Finally we let $N$ go to infinity to get :
$$d_\infty(y,\psi(x))\leq \epsilon.$$

This concludes the proof of the lemma as $\epsilon>0$ may be chosen arbitrarily small.
\end{proof}

\subsubsection{Dynamical consequences}
We let now  $T$ be an invertible map acting  continuously on $X$.   For a Borel map $\xi:X\rightarrow \Lambda^\mathbb{Z}$ we let $\xi^*:\mathcal{M}(X,T)\rightarrow \mathcal{M}(\Lambda^\mathbb{Z},\sigma)$ be the map induced by $\xi$ on the set of invariant Borel probability measures. As done for the full shift in the previous subsection we denote  for any $x\in X$ by $\phi(x)\subset \mathcal{M}(X,T) $ the set  of limits of empirical measures at $x$, i.e. the set of  accumulation points of the sequence $(\frac{1}{n}\sum_{0\leq k<n}\delta_{T^k x})_n$ for the weak-* topology.

\begin{lemma}\label{cvu}
Let  $\psi:X\rightarrow \Lambda^\mathbb{Z}$  be a uniform limit with respect to $d_\infty$ of $(\psi_k)_k$. Assume moreover $\psi$ and $(\psi_k)_k$ are equivariant Borel  maps. 

Then the induced maps $(\psi_k^*)_k$ converge uniformly to $\psi^*$ with respect to $d_*$.
\end{lemma}

\begin{proof}
 By the ergodic decomposition it is enough to consider only ergodic $T$-invariant measures $\mu$. By Birkhoff ergodic theorem we have for $\mu$-almost every $x$ and for all finite cylinders $A$ in $\Lambda^\mathbb{Z}$ :

$$\mu(\psi^{-1}A)=\lim_n\frac{1}{2n -1} \sum_{|l|<n}\delta_{T^lx}(\psi^{-1}A).$$

By equivariance we have

\begin{eqnarray*}
\mu( \psi^{-1}A) &= & \lim_n\frac{1}{2n -1} \sum_{|l|<n}\delta_{\psi\circ T^l(x)}(A),\\
&= & \lim_n\frac{1}{2n -1}\sum_{|l|<n}\delta_{\sigma^l \circ \psi (x)}(A).
\end{eqnarray*}

Since this holds for any  $A$ the sequence of empirical measures \\ $(\frac{1}{2n -1}\sum_{|l|<n}\delta_{\sigma^l \circ \psi (x)})_n$ is converging to $\psi^*\mu$ in the weak-*topology, in others terms we have $\{\phi\circ\psi(x)=\psi^*\mu\}$ for  $x$ in a set $E^\mu$ of full $\mu$-measure. Similarly for any integer $k$ we have $\phi\circ\psi_k(x)=\{\psi_k^*\mu\}$ for $x$ in a set $E_k^\mu$ of full $\mu$-measure.

 As by Lemma \ref{emp} the sequence $(\phi\circ \psi_k)_k$ converges uniformly to $\phi\circ \psi$ with respect to $d_H$   we conclude that $\psi_k^*$ converges to $\psi^*$ uniformly with respect to $d_*$. Indeed for any $\mu$ we take $x_\mu \in E^\mu\cap \bigcap_kE^\mu_k$ and  we conclude that:

\begin{eqnarray*}
\sup_\mu d_*(\psi_k^*\mu, \psi^*\mu)&= & \sup_{x_\mu, \mu}d_*(\phi\circ \psi_k(x_\mu),\phi\circ \psi(x_\mu)),\\
&\leq &  \sup_{x}d_H(\phi\circ \psi_k(x),\phi\circ \psi(x)),\\
& \xrightarrow{k\rightarrow +\infty}0. 
\end{eqnarray*}
\end{proof}

\begin{rem}The above lemma may also be proved by using the Ornstein's $d$-bar distance. Indeed it follows from Theorem I.9.10 in \cite{Sh} that for any $\mu\in \mathcal{M}(X,T)$ the $d$-bar distance of $\psi^*\mu$ and $\psi^*_k\mu$ is less than $\sup_{x\in X}d_\infty(\psi_k x, \psi(x) )$.  Moreover it is well known the $d$-bar distance is stronger than the weak-star topology. 
\end{rem}

Assuming moreover continuity of the maps $(\psi_k)_k$  we prove the identity $\psi^*\phi(x)=\phi\circ \psi(x)$ for every $x\in X$. 

\begin{lemma}\label{bnb}Let $\psi:X\rightarrow \Lambda^\mathbb{Z}$  be a uniform limit with respect to $d_\infty$ of $(\psi_k)_k$. Assume moreover $\psi$  (resp. $(\psi_k)_k$) are  equivariant Borel (resp. continuous) maps. 

Then for all $x\in X$, 
$$\psi^*\phi(x)=\phi\circ \psi(x).$$
\end{lemma}

\begin{proof}
The maps $\psi_k$ being continuous  the associated maps $\psi_k^*$ induced  on the set of Borel probability measures on $X$ endowed with the weak-* topology are also continuous and therefore we have the equality 
$\psi_k^*\phi=\phi\circ\psi_k$. By the above Lemma \ref{cvu}, the left member goes uniformly  to $\psi^*\phi$ (with respect to $d_H$), whereas the right member  goes uniformly to $\phi\circ \psi$ when $k$ goes to infinity according to Lemma \ref{emp}. 
\end{proof}

\begin{lemma}\label{sssup}
Let $\psi:X\rightarrow \Lambda^\mathbb{Z}$  be both a pointwise limit with respect to $d$ and a uniform limit with respect to $(d_N)_N$  of $(\psi_k)_k$.  Assume moreover the maps $(\psi_k)_k$ are 
 continuous equivariant. 
 
 Then any $\sigma$-invariant measure $\mu$ on $\overline{\psi(X)}$ is supported on $\psi(X)$, i.e. $\mu(\psi(X))=1
$.
\end{lemma}

\begin{proof}
Let $\mu$ be a $\sigma$-invariant ergodic measure on $\overline{\psi(X)}$.  According to Birkhoff ergodic theorem there is $y\in \overline{\psi(X)}$ with $\phi(y)=\{\mu\}.$ By Lemma \ref{sat} one can find $x\in X$ such that $d_\infty(y,\psi(x))=0$. As previously discussed after Lemma \ref{emp} it implies that
$\phi\circ\psi(x)=\phi(y)=\mu$.  Finally by Lemma \ref{bnb} we have $\psi^*\phi(x)=\phi\circ \psi(x)=\mu$. This concludes the proof as   $\psi^*\phi(x)$ is supported on $\psi(X)$.
\end{proof}

\subsection{Proof of Propostion \ref{zer}}

We first deal with aperiodic systems. 
For general systems with periodic points the proof is a little more involved as we need to encode periodic points. In this last case we  will adapt the construction of Krieger \cite{Kri} (Theorem \ref{kkkk}) using the $per$-expansiveness property and the marker property given in  Lemma  \ref{mar}.\\

Observe first that in any case one only needs to embed our system in a full shift  with some finite alphabet. Indeed the topological entropy and the cardinality of periodic points with least period $n$ for any $n$ will be preserved by our Borel embedding so that by applying Krieger Theorem for topological expansive systems (Theorem \ref{kkkk})  we get after a composition an embedding in the shift space with the desired number of letters (but also in any subshift of finite type with the suitable lower bound on the topological entropy and cardinality of periodic points).

\subsubsection{Construction of the embedding, the aperiodic case}
We consider an aperiodic  asymptotically $h$-expansive zero-dimensional dynamical system  $(X,T)$. Let $K$ be an integer with $\log K > h_{top}(T)$. We let $\alpha>0$ be the difference $\alpha=\log K-h_{top}(T)$.  \\

\textit{0. Dynamical Markers.}
 We let $(\mathcal{V}_k)_k$ be a sequence of nested clopen partitions as in Lemma \ref{hex}. 
We let $n_1$ be such that for $n>n_1$: 
$$\sharp \mathcal{V}_1^{n}\leq e^{n(h_{top}(T)+\alpha/4)}< K^{n(1-\alpha/2)-1}$$
and for $k>1$ we let $n_k$ be such that  for $n\geq n_k/2$:
\begin{eqnarray*}
\sup_{U^n\in \mathcal{V}_{k}^{n}}\sharp\{ V^n, \ V^n\in \mathcal{V}_{k+1}^{n} \textrm{ with } V^n\subset  U^n\}& < & K^{n\alpha/2^k-1 }.
\end{eqnarray*}

We may also ensure that the sequence  $(n_k)_k$ satisfies $\alpha n_k>>2^k$ and $n_k>2n_{k-1}$ for all integers $k$.  We will build an injective map from $X$ to $\{|,+,\|,\circ,1,...,K\}^\mathbb{Z}$ which conjugates $T$ to the shift  as in the  ergodic generator Krieger theorem \cite{Kri} (Theorem \ref{kkk} of the Introduction) by first encoding the dynamic with respect to the covers $\mathcal{V}_1, \mathcal{V}_2$, ... ,$\mathcal{V}_k$,... on finite pieces of orbits of length of order  $n_1,n_2,...,n_k,...$. 
These pieces of orbits are given by the  topological Rokhlin  towers $(U_k)_k$ given by Lemma \ref{mar} with respect to the sequences $(n_k)_k$. More precisely, we will consider the following nonincreasing sequence  of $\mathbb{Z}$-subsets $(G_k(x))_k$ for any $x\in X$. We first let $G_{1}(x):=  \{ l\in \mathbb{Z}, \  f^lx\in U_1\}$. Then we define for any integer $k>1$ the subset $G_k(x)$ of $\mathbb{Z}$ by induction on $k$ as  the subset of  integers $i\in  G_{k-1}(x)$, such that if $j$ denotes the successor of $i$ in $G_{k-1}$ then there is $l\in [i,j[$ with $f^lx\in U_k$ (see Figure \ref{bl}). Observe that the distance between two consecutive integers in $G_k(x)$ has length larger than or equal to $n_k/2$ and less than $4n_k$.  We let $\underline{n}$ be the sequence of integers $(n_k)_k$ and $G(x):=(G_k(x))_k$ be the nonincreasing sequence of the subsets $G_k(x)$. \\

\begin{figure}[!ht]
\vspace{-0,5cm}
\includegraphics[scale=0.2]{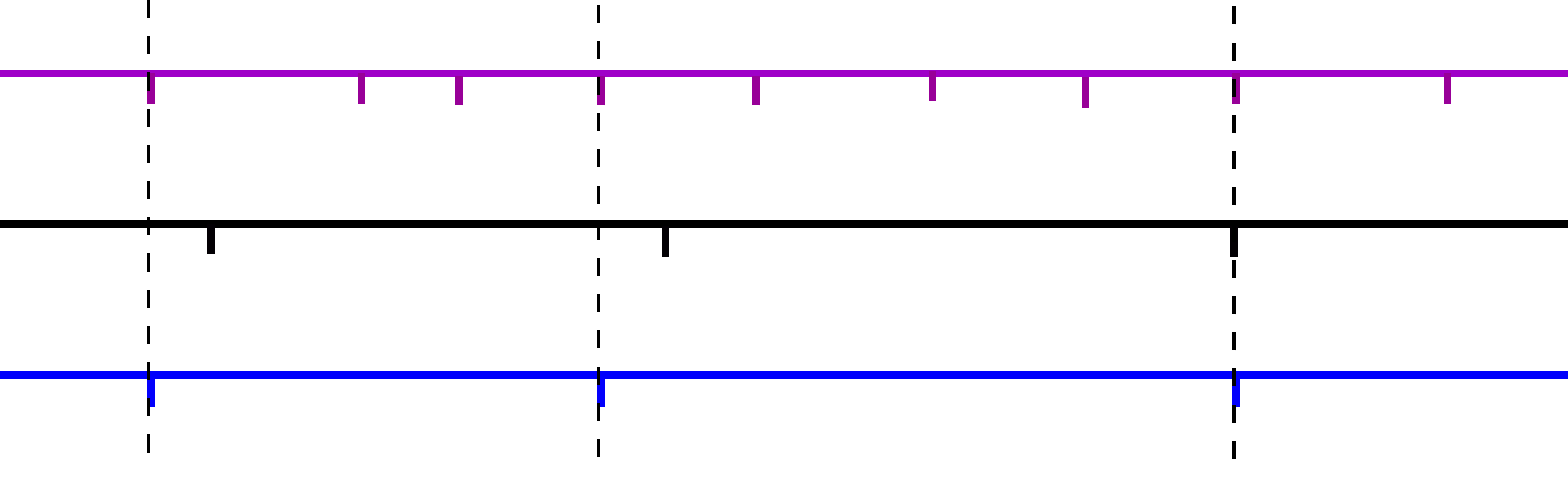}
\centering
\caption{\label{bl}\textbf{The set $G_k(x)$ from the return times in $U_k$.} The first line and dashes in purple represent the element in $G_{k}(x)$. On the second line we draw the 
times in $U_{k+1}$. The last line gives the associated times in $G_{k+1}(x)$.}
\end{figure}

\textit{1. Marker structure.} 
 We consider the sequence $G=G(x)$ for some fixed $x$. For simplicity we also denote by  $G_k$ the set $G_k(x)$.  For any integer $k>0$ an empty $k$-block  associated to $G$ will be  a finite word $u^k$ in $\{|, +,\|,*,\circ \}$ indexed with respect to $G_k$ : it starts and finishes at consecutive integers $i_k<j_k\in G_k$, that is $u^k:=(u_{i_k},...,u_{j_k-1})$. The length $j_k-i_k$ of such a block $u^k$ is denoted by $|u^k|$. The empty $k$-blocks  and their $k$-marker, $k$-filling, $k$-free positions are  then defined by induction  as follows. The exact meaning of the symbols in the alphabet will be given later on, but in order to help the reader we roughly explain it now:
\begin{list}{}{}
 \item \_ \textbf{$|$} and \textbf{$\|$}  are marking the times in $G_k$,
\item  \_ \textbf{$+$}  point out the lack of a marker $\|$,
\item   \_ \textbf{$*$} represent the filling positions, which are  the places where will be written the code of finite orbits,
\item \_ \textbf{$\circ$} are called the free positions. They correspond to the remaining positions : we still did not assign them any value at a given  step of the construction. 
 \end{list}

 Let us be more precise now.  Firstly a empty $1$-block is a  word $u^1$  such that:

\begin{itemize}
\item the  first coordinate of $u^1$, the $1$-marker position, coincides with $|$,
\item the next  $[(1-\alpha/2)|u^1|]$ coordinates\footnote{$[x]$ denotes the integer part of a real number $x$} of $u^1$, called the $1$-filling positions,  all take the value $*$,
\item the remaining coordinates, called the $1$-free positions,  all take the value $\circ$. The first free position in $u_1$ is called the  $2$-marker position of $u_1$.
\end{itemize}

For $k>1$ an empty $k$-block $u^k$ is obtained from a concatenation of consecutive empty $(k-1)$-blocks, where we change  $\circ$ to $\|$  at the $k$-marker  position of the first $(k-1)$-block. This position in $u^k$ defines the $k$-marker position. Moreover we change  the $k$-marker positions  of the others $(k-1)$-blocks in $u^k$ to $+$. 
These positions  allow to detect the first coordinate $i_k\in G_k$ of the $k$-block $u^k$. The $[\alpha  |u^k|/2^k]$ next $(k-1)$-free positions are the $k$-filling positions of $u^k$ and the remaining $(k-1)$-free positions of the concatenation define the $k$-free positions of $u^k$ (see Figure \ref{ffi} below). The  $(k+1)$-marker position of $u^k$ is the first $k$-free position in $u^k$.\\
\begin{figure}[!h]
\vspace{-1,5cm}
\includegraphics[scale=0.22]{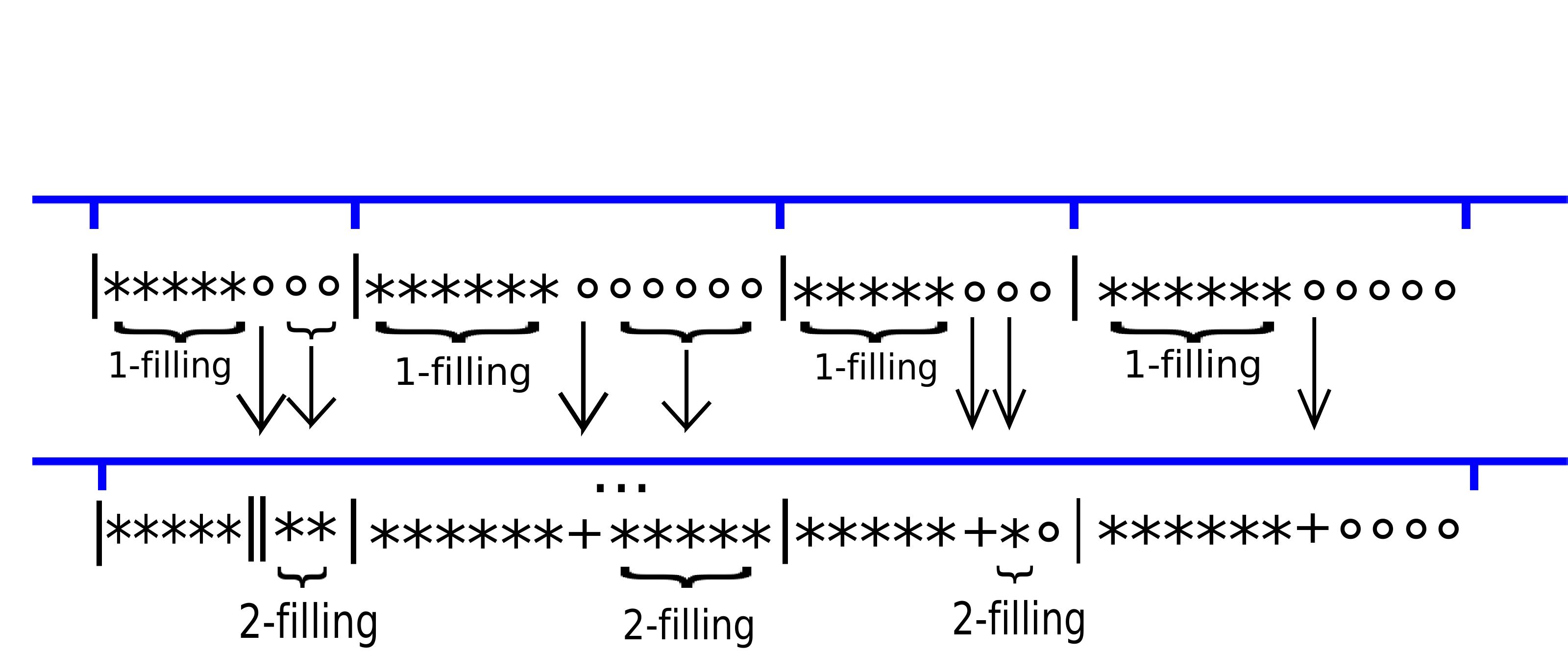}
\centering
\caption{\label{ffi}\textbf{Empty $1$- and $2$-blocks.} For $k=1,2$ the empty $k$-blocks and their  marker, filling and free positions are represented on the $k^{th}$ blue line. Elements in $G_k$ are given by the blue dashes.}
\end{figure} 

The empty $k$-code $\psi^\emptyset_k$ with respect 
to $G$ is then just the infinite word in $\{|,+, \|,*,\circ \}^\mathbb{Z}$ obtained by concatenation of  all empty $k$-blocks. \\

\textit{2. First scale encoding.}  let $\tau_1$ be the first return time of $y$ in  $U_1$,  that is 
$\tau_1(y):=\inf\{n>0, \ T^ny\in U_1\}$. For any $n\geq n_1$ we consider an injective map  from $\mathcal{V}_1^n$ to the set of words of length $[(1-\alpha/2)n]$ with letters in $\{1,...,K\}$. Then to any $y\in U_1$ we associate such a  word corresponding to $V_1^{\tau_1(y)}(y)$. The $1$-code $\psi_1(x)$ of $x$ is then obtained by replacing 
 in the empty $1$-blocks in  $\psi_1^\emptyset(x)$ associated to  $G(x)$ the symbols $*$'s in the $1$-filling positions by these words for $y=T^lx$ with $l$ being the  return times in $U_1$, which corresponds to the first index of the empty $1$-blocks. \\

\textit{3. Higher scales encoding.} We will now encode the dynamics with respect to $\mathcal{V}_2$ conditionally to $\mathcal{V}_1$.  For any $n\geq n_2/2$ and for any $V_1^n\in \mathcal{V}_1^n$ we consider an injective map  from $\{V_2^n\in \mathcal{V}_2^n, \ V_2^n\subset V_1^n\}$ to the $[\alpha n/4]$-words of $\{1,...,K\}^\mathbb{Z}$. Then in an empty $2$-block of size $n$  of $\psi_2^\emptyset(x)$ we replace the symbols   $*$'s  at the $2$-filling positions by the  $[\alpha n/4]$-words of $\{1,...,K\}^\mathbb{Z}$ associated to $V_2^{n}(y)$ conditionally to $V_1^{n}(y)$ for $y=T^lx$ with $l$ the first index of the empty $2$-blocks. In a similar way we build $\psi_k$ for any $k>2$. The sequence $(\psi_k(x))_k$ is converging pointwisely for the product topology in $\{1,...,K\}^\mathbb{Z}$ as for any $i$ the $i^{th}$ coordinate of $\psi_k(x)$ is constant after some rank. We let $\psi(x)$ be the pointwise limit of $\psi_k(x)$. \\

\textit{4. Decoding.}   
The  $k$-code $\psi_k(x)$ of $x$ may be deduced  from 
$\psi(x)$ by replacing the $l$-marker and $l$-filling positions in $\psi(x)$ for $l>k$ by $k$-free positions. Thus we get in this way  a sequence of continuous maps $\pi_k:\psi(X)\rightarrow \{1,...,K\}^\mathbb{Z}$  such that $\pi_k\circ \psi(x)=\psi_k(x)$. We can then identify $\mathcal{V}_1^{n}(x)$ for any integer $n$ by reading the  words in the $1$-filling positions of $\psi_1(x)$  and finally $\mathcal{V}_l^{n}(x)$ inductively on $l$  by reading in $\psi_l(x)$ the words in the $l$-filling positions. Consequently  any fiber of $\psi_k$ is contained in a unique element of $\mathcal{V}_k^\infty=\{\bigcap_{n\in \mathbb{Z}}T^{-n}V_k^n, \ V_k^n\in \mathcal{V}_k \}$. In particular $\psi_k$ is $\epsilon_k$-injective with $\epsilon_k$ being the diameter of $\mathcal{V}_k$. As the diameter of $\mathcal{V}_k$ goes to zero as $k$ goes to infinity it  follows that the limit $\psi$ is injective. \\

\textit{5. Continuity and convergence  of $(\psi_k)_k$.} The maps $(\psi_k)_k$ are continuous because $(\mathcal{V}_k)_k$ are open covers and $U_k$ are clopen sets. Moreover $\psi_k(x) $ and $\psi(x)$  differs on any $k$-block only at the $k$-free positions.   Recall that  these positions represent a proportion of at most $\alpha/2^k$ positions in $k$-blocks. Moreover for any $n>n_k^2$ the segment $[-n,n]$ is the union of at most $2(2n+1)/n_k$  $k$-blocks and $2$ $k$-subblocks (whose length is less than or equal to $4n_k$). 
Then we have
$$\sharp \{i,\ |i|\leq n \text{ and }  \ (\psi_k)^i(x)\neq (\psi)^i(x)\} \leq (2n+1)\alpha/2^k +2(2n+1)/n_k+8n_k.$$
  
Thus we conclude that  $d_N\left(\{i, \ (\psi_l)^i(x)\neq (\psi)^i(x) \} \right)\lesssim \alpha/2^k$  for any $x\in X$, any $N>n_k^2$ and any $l\geq k$  because we took $\alpha n_k>>2^k$ and $\{i, \ (\psi_l)^i(x)\neq (\psi)^i(x) \}\subset \{i, \ (\psi_k)^i(x)\neq (\psi)^i(x) \}$. Therefore $\psi_k$ goes uniformly to $\psi$ w.r.t. $(d_N)_N$ when $k$ goes to infinity. \\

\textit{6.  Principal strongly faithful symbolic extension as  the inverse of Krieger embedding.}
The inverse of $\psi$ is (uniformly) continuous on $\psi(X)$. Indeed if $\psi(x)$ and $\psi(y)$ coincide on their $[-4n_k,4n_k]$  coordinates then  they belong to the same $l$-block for any $l\leq k$ and in particular $x$ and $y$ belong to the same element of $\mathcal{V}_k$. 

Lemma \ref{sssup} applies to $\psi$, thus any $T$-invariant measure on $\overline{\psi(X)}$ is supported on $\psi(X)$.
 Finally as $\psi:(X,T)\rightarrow (\psi(X),\sigma)$ is a Borel isomorphism, the induced map on the set of invariant measures is bijective and preserve the measure theoretical entropy. This proves $\pi$ is  faithful and principal.

\begin{rem}
Observe that $\psi(X)$ is in general not closed. Indeed let $(X,T)$ be the odometer to base $(p_k)_k=(2^k)_k$. For any positive integer $n$ we let  $x_n$ be the point in $X$ given by $x_n=(...,0,...0,2^n,0,...)$ where $2^n$ is at the $(n+1)^{th}$ coordinate and  $U_n$ be the clopen set of points whose $n$ first coordinates are zero.  The sequence $(U_n)_n$ satisfies the properties of the Marker Lemma (Lemma \ref{mar}). Then we consider the Borel embedding $\psi: X\rightarrow \{0,1\}^\mathbb{Z}$ given by the previous construction.  Let $y$ be a limit point of $(\psi(x_n))$. If it belongs to $\psi(X)$ the point $y$ is  necessarily $\psi(0)$ as $\pi_k(y)=\lim_n\pi_k\circ \psi(x_n)=\lim_n\psi_k(x_n)=\psi_k(0)$. Now the $-1$ position of $\psi(0)$ is a $n+1$-free position for all $n$ whereas for any $n$ the $-1$ position of $\psi(x_n)$ is a $n+1$-filling position (indeed the zero$^{th}$ coordinate  is in the middle of an $n+1$-block of $\psi(x_n)$ and at the beginning of an $n+1$ block of $\psi(0))$.
\end{rem}

\subsubsection{Construction of the embedding, the periodic case}
We consider now the general case, i.e. we consider asymptotically expansive zero-dimensional systems not necessarily aperiodic ones. Here we will distinguish among $k$-blocks regular ones and singular ones, which encode respectively the dynamics far from or close to the periodic points. 
Each singular block is the repetition of a brick, which is a finite word encoding the periodic point associated to the singular block (its length is in particular equal to the period). We let again $\alpha=\log K-h_{top}(T)>0$.\\

\textit{0. Dynamical Markers.} We consider a sequence of clopen partitions $(\mathcal{V}_k)_k$ given by Lemma \ref{hex} and we let  $\underline{n}=(n_k)_k$  be a sequence of integers satisfying for any $k$ and for all $n>n_k/2$   

$$\sup_{U^n\in \mathcal{V}_{k}^{n}}\sharp \{V^n, \ V^n\in \mathcal{V}_{k+1}^{n} \textrm{ with } V^n\subset  U^n\} \leq  e^{n\alpha/2^k }$$

 and  moreover

\begin{eqnarray}\label{peexp}\sup_{V_k^{n}\in \mathcal{V}_k^{n}}\log\{\sharp Per_{n}(X,T)\cap V_k^{n}\}& <&\alpha/2^k.
\end{eqnarray}

We also choose $n_1$ large enough so that for $n\geq n_1$ we have $\sharp Per_{n}(X,T)<K^{n-1}$  and $\sharp \mathcal{V}_1^{n}<K^{n(1-\alpha/2)-1}$. Moreover we may  ensure that $(n_k)_k$ satisfy  $\alpha n_k>>2^k$ and $n_k>16n_{k-1}$ for any $k\in \mathbb{N}$ (we let $n_0=0$).  Far from the  $n_1,n_2,...,n_k,...$-periodic points we will encode the dynamics with respect to the covers $\mathcal{V}_1, \mathcal{V}_2$, ... ,$\mathcal{V}_k$,... on finite pieces of orbits of length of order $n_1,n_2,...,n_k,...$. These pieces of orbits are given again by the  topological Rokhlin  towers over $(U_k)_k$ given in Lemma \ref{mar} with respect to the sequences $\underline{n}=(n_k)_k$ and $(\epsilon_k)_k$. This last  sequence will be precised latter on.
But we already  assume $\epsilon_k$ so small that any consecutive return times  in $U_k$ with distance larger than or equal to  $2n_k$ corresponds to the visit of a $\epsilon_k$-neighborhood of a single periodic orbit. We may also assume that three consecutive return times in $U_k$, which are at least  $2n_k$-far from each other, are associated to the same periodic orbit. 

For any periodic orbit with least period  $n\leq n_1$ we choose arbitrarily one point $x$ in the orbit. For such a periodic point we let $G_1(f^lx)=F_1(f^lx)=l+n\mathbb{Z}$ and $E_1(x)=\emptyset$ for any integer $l$. Let us now define $G_1(x),F_1(x),E_1(x)$ for aperiodic or $n$-periodic points $x$ with $n>n_1$. 
Let $l$ be a return time of $x$ in $U_1$. We distinguish three cases: 
\begin{itemize}
\item there are other return times  in $U_1$, say  $k,m$,  with $l-4n_1<k<l<m<l+4n_1$. We let in this case $l'=l$,
\item there is a return time $k$ in $U_1$   with $l-4n_1<k<l$ (resp. $l<k<l+4n_1$), but no return times $p$ with   $l<p<l+4n_1$ (resp. $l-4n_1<p<l$), 
in this case  $T^{l+n_1}x,T^{l+n_1+1}x,..., T^{l+2n_1}x$ (resp. $T^{l-2n_1}x,T^{l-2n_1+1}x,...,T^{l-n_1}x$) belong to the $\epsilon_1$-neighborhood of a single periodic orbit with period $m$ less than or equal to $n_1$.  Then we let $l'$ be the least  integer   in $[l,l+n_1[$ (resp. largest integer in $]l-n_1,l]$) such that $T^{l'+cm}x$ is $\epsilon_1$-close to the distinguished  point in this periodic orbit where $c$ is a positive (resp. negative) integer with $l+n_1\leq l'+cm\leq l+2n_1$ (resp. $l-2n_1\leq l'+cm\leq l-n_1$).
\item there is no return time $k$ in $U_1$ with $|k-l|< 4n_1$.  
\end{itemize}

The set $E_1(x)$ is then defined as the set of integers $l'$ as defined above for return times $l$ in $U_1$ of the two first kinds. Observe two consecutive integers in $E_1(x)$ are at least $n_1$-far from each other. 
Now if $(k',l')$ are two consecutive integers in $E_1(x)$ with $l'-k'>4n_1$ then $l'$ and $k'$ are of the second kind and  $l'-k'\in m\mathbb{Z}$ where $m$ is the period of the single periodic orbit close to the piece of orbit $f^{k'}(x),...,f^{l'}(x)$. Then we let $F_1(x)$ be the union of $[k',l'[\cap \left(k'+m\mathbb{Z}\right)$ over all such pairs $(k',l')$  (see Figure \ref{zoo}). Thus two consecutive integers in the union $G_1(x)=E_1(x)\cup F_1(x)$ are at most $4n_1$-far from each other. As the set $U_1$ is clopen the maps $x\mapsto E_1(x), F_1(x),G_1(x)$ are continous in the following sense : for any bounded interval of integers $I\subset \mathbb{Z}$  and for all $x\in X$ the sets $E_1(y)\cap I, F_1(y)\cap I$ and $G_1(y)\cap I$ are constant for $y$ in  a neighborhood of $x$.  

\begin{figure}[!ht]
\vspace{-0.7cm}
\hspace{-2,1cm}
\includegraphics[scale=0.21]{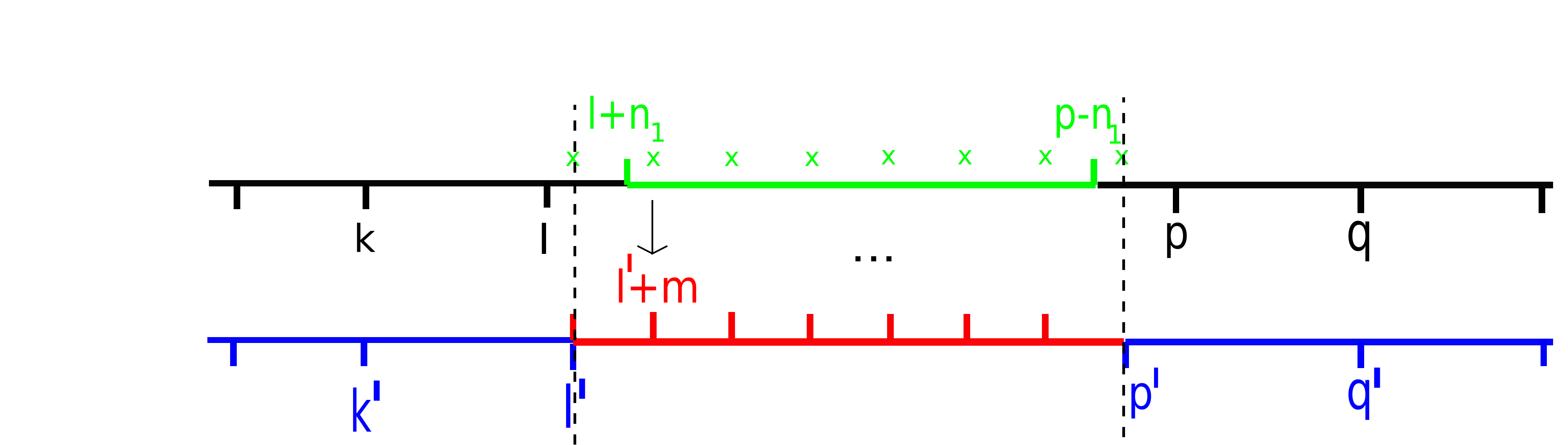}
\captionsetup{width=0.8\textwidth}
\caption{\label{zoo}\textbf{The set $G_1(x)$ from the return times in $U_1$.}   On the first line the part in black corresponds to times in $\bigcup_{|l|<n_1}T^{l}U_1$ whereas the  green part are times in the $\epsilon_1$-neighborhood of $Per_{n_1}$. Black dashes are times in $U_1$ and the green crosses correspond to the visits of the $\epsilon_1$-neighborhood of the distinguished point of the periodic orbit. On the second line the elements of $E_1(x)$ and $F_1(x)$ are respectively represented by blue and red dashes.}
\end{figure}

The set $E_k(x),F_k(x)$ are then defined by induction on $k>1$ in a similar way as follows. We assume the sets $E_{k-1}(x),F_{k-1}(x)$ already built  and they satisfy : 
\begin{itemize} 
\item two consecutive integers in $G_{k-1}(x)=E_{k-1}(x)\cup F_{k-1}(x)$ are at most $5n_{k-1}$-far from each other,
\item the maps $x\mapsto E_{k-1}(x),F_{k-1}(x)$ are continuous.
\end{itemize}

Let us define now $E_k(x)$ and $F_k(x)$. Firstly for a $n$-periodic point $x$ with $n\leq n_k$ either $n\leq n_{k-1}$ and we let $G_k(x)=F_k(x)=F_{k-1}(x)$, $E_k(x)=\emptyset$  or $n_{k-1}<n\leq n_k$. 
  In this last case we choose a point, say $f^l(x)$, in the orbit of $x$ with $l\in G_{k-1}(x)$ and finally we let $G_k(x)=F_k(x)=l+n\mathbb{Z}$ and $E_k(x)=\emptyset$. 
 
 For other points we classify return times in $U_k$ and define $E_k(x)$ and $F_k(x)$ in a similar way as for $k=1$. Firstly observe that we can take $\epsilon_k$ so small that for any $x\in Per_{n_k}^{\epsilon_k}$ the set $G_{k-1}(x)$ coincides with the set $G_{k-1}(y)=E_{k_1}(y)$ on $[-n_k,n_k]$ where $y$ is the periodic point which is $\epsilon_k$-closed to $x$. We consider $l$ be a return time of $x$ in $U_k$. We distinguish three cases: 
\begin{itemize}
\item there are other return times  in $U_k$, say  $j,m$,  with $l-4n_k<j<l<m<l+4n_k$. We let in this case $l'$ be the largest integer in $G_{k-1}(x)$ less than $l$ (by the induction hypothesis we have $|l'-l|<5n_{k-1}$),
\item there is a return time $j$ in $U_k$   with $l-4n_k<j<l$ (resp. $l<j<l+4n_k$), but no return time $p$ with   $l<p<l+4n_k$ (resp. $l-4n_k<p<l$), 
in this case  $T^{l+n_k}x,T^{l+n_k+1}x,..., T^{l+2n_k}x$ (resp. $T^{l-2n_k}x,T^{l-2n_k+1}x,...,T^{l-n_k}x$) belong to the $\epsilon_k$-neighborhood of a single periodic orbit with period $m$ less than or equal to $n_k$.  Then we let $l'$ be the least  integer in $G_{k-1}(x)$  larger  (resp. largest integer less) than $l$ such that $T^{l'+cm}x$ is $\epsilon_k$-close to the distinguished  point in this periodic orbit where $c$ is a positive (resp. negative) integer with $l+n_k\leq l'+cm\leq l+2n_k$ (resp. $l-2n_k\leq l'+cm\leq l-n_k$).
\item there are no return times $j$ in $U_k$ with $|j-l|< 4n_k$.  
\end{itemize}

The set $E_k(x)$ is then defined as the set of integers $l'$ as defined above for return times $l$ in $U_k$ of the two first kinds. Observe two consecutive integers in $E_k(x)$ are at least $n_k/2$-far from each other. 
Now if $(k',l')$ are two consecutive integers in $E_k(x)$ with $l'-k'>4n_k$ then $l'$ and $k'$ are of the second kind and  $l'-k'\in m\mathbb{Z}$ where $m$ is the period of the single periodic orbit close to the piece of orbit $f^{k'}(x),...,f^{l'}(x)$. Then we let $F_k(x)$ be the union of $[k',l'[\cap \left(k'+m\mathbb{Z}\right)$ over all such pairs $(k',l')$  (see Figure \ref{zoe}). 

\begin{figure}[!h]
\hspace{-2,1cm}
\includegraphics[scale=0.21]{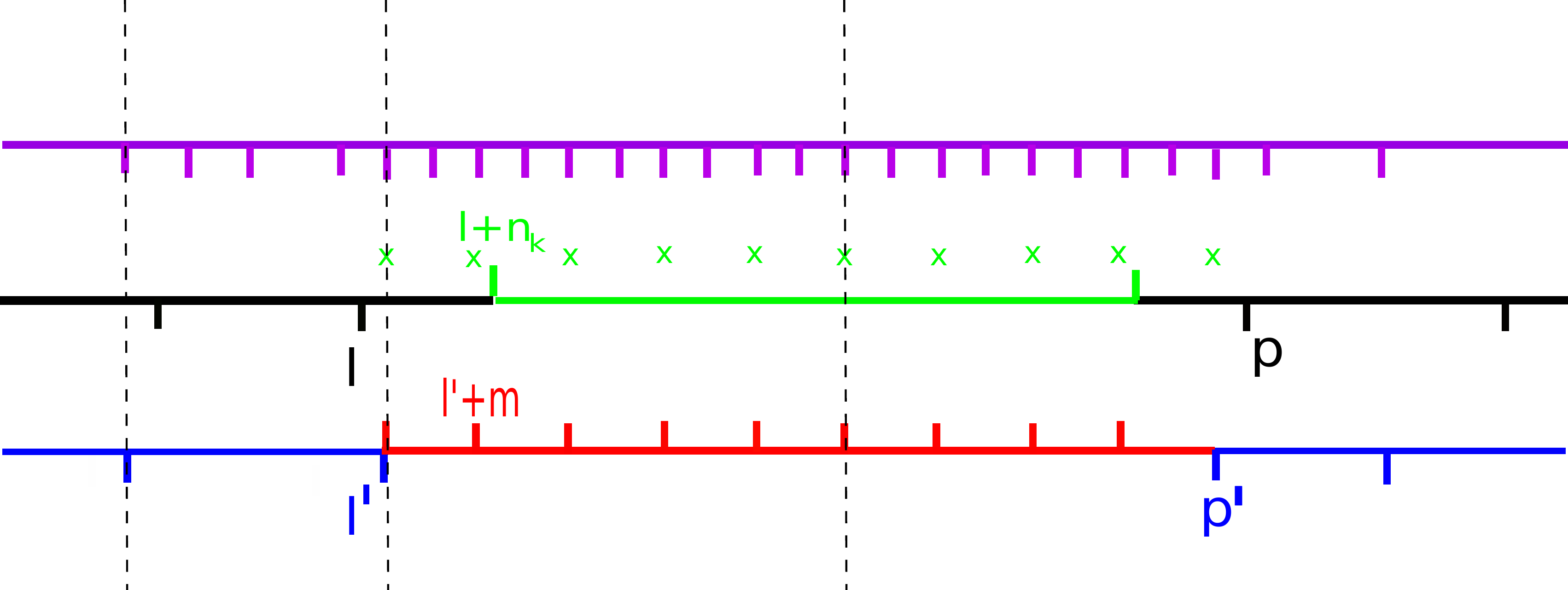}
\captionsetup{width=0.8\textwidth}
\caption{\label{zoe}\textbf{The set $G_k(x)$ from $G_{k-1}(x)$ and the return times in $U_k$.}   On the first line in purple is represented the times in $G_{k-1}(x)$. The second line corresponds, as the first line in Figure \ref{zoo}, to the return times in $U_k$ (in black) and in the $\epsilon_k$-neighborhood of  $Per_{n_k}$ (in green).  On the last line we draw the elements of $E_k(x)$ and $F_k(x)$ with  respectively  blue and red dashes.}
\end{figure}
Thus two consecutive integers in the union $G_k(x)=E_k(x)\cup F_k(x)$ are at most $5n_k$-far from each other.  The maps $x\mapsto E_k(x), F_k(x)$ are continuous. Moreover it follows  from the construction that  $G_{k}(x)\subset G_{k-1}(x)$ for all integers $k$ and for all $x\in X$. \\

\textit{1. Marker structure.}
We consider the nonincreasing sequence $G=(G_k)_k$  of subsets of $\mathbb{Z}$ associated to a given point $x\in X$. Recall that for any $k$, the set $G_k$ is the union of the two  subsets $E_k$ and $F_k$ with the following properties. Moreover any two consecutive  $p<q$ in $E_k\cup\{-\infty,+\infty\}$ satisfy :
\begin{itemize}
\item either $n_k/2\leq q-p\leq 4n_k$, then there are no element of $F_k$ in the interval $[p,q[$. Any word over the interval of integers $[p,q[$ will be called a regular $k$-block (see Figure \ref{f} below).
\item  or $q-p>4n_k$, then there exists an integer $0<m\leq n_k$ such that $q-p\in m\mathbb{Z}\cup \{+\infty\}$ and $F_k\cap [p,q[=\left(p+m\mathbb{Z}\right)\cap [p,q[$ (resp. $F_k\cap [p,q[=\left(q+m\mathbb{Z}\right)\cap [p,q[$, resp. $F_k=n+m\mathbb{Z}$ for some integer $n$) when $p$ is finite (resp. $p=-\infty$ and $q$ is finite, resp. $p=-\infty$ and $q=+\infty$). Any word over the interval of integers $[p,q[$ will be called a singular $k$-block, whereas words over intervals of the form $[p+km, p+(k+1)m[$ with $k\in \mathbb{N}$ and $p+(k+1)m\leq q$ (resp. $[q+(k-1)m, q+km[$ with $-k\in \mathbb{N}$, 
resp. $[n+km,n+(k+1)m[$ with $k\in \mathbb{Z}$) are said to be $l$-bricks with $n_{l-1}<m\leq n_l$. 
\end{itemize}

\begin{figure}[!ht]
\includegraphics[scale=0.22]{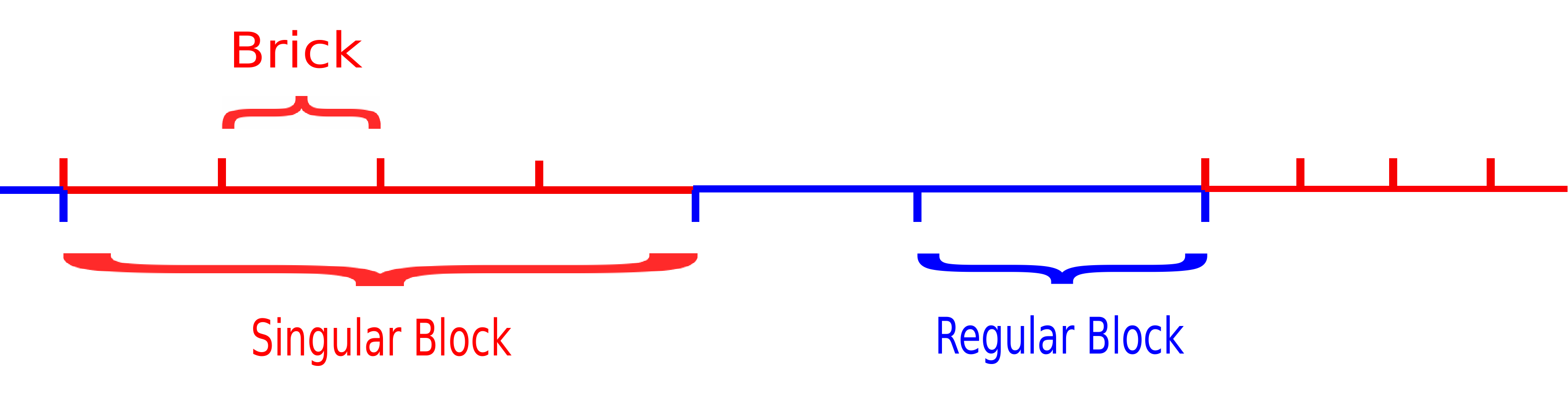}
\captionsetup{width=0.8\textwidth}
\centering
\caption{\label{f}\textbf{$k$-blocks and $l$-bricks with $l\leq k$ for a given $k$.}   
Elements in $F_k$ are given by red dashes and elements in $E_k$ by blue ones. Regular $k$-blocks (resp. bricks) are the blue intervals (resp. red) between two consecutive blues (resp. red) dashes. A singular $k$-block is a maximal red interval : it is the self concatenation of a single $l$-brick with $l\leq k$.}
\end{figure}

The empty  $k$-blocks and $k$-bricks are words in the alphabet $\left\{*,\times,\circ,],[,+,(,),|, \times,\times|,\times ],\times [\right\}$ over $k$-blocks and $k$-bricks respectively. They    are  defined  by induction on $k$, together with  their marker, filling, free positions,  as follows. \\

\underline{Empty $1$-blocks.} Firstly a empty regular $1$-block is a  regular $1$-block $u^1$  such that  \begin{itemize}\item the  first and last letter of $u^1$, resp. called  the left and right $1$-marker positions, coincide respectively with $($ and $)$,
\item the next  $[(1-\alpha)|u_1|]$ coordinates of $u^1$, called the $1$-filling positions  all take the   value $*$,
\item the remaining coordinates, called the $1$-free positions,  all take the value $\circ$.
\end{itemize}

The $2$-marker position of a $1$-regular block is the first $1$-free position of this block. An empty $1$-brick is a $1$-brick  whose first letter   $\times$ represents the marker position and  other ones $*$ are $1$-filling positions.  \\

\begin{figure}[!h]
\hspace{-1,1cm}
\includegraphics[scale=0.18]{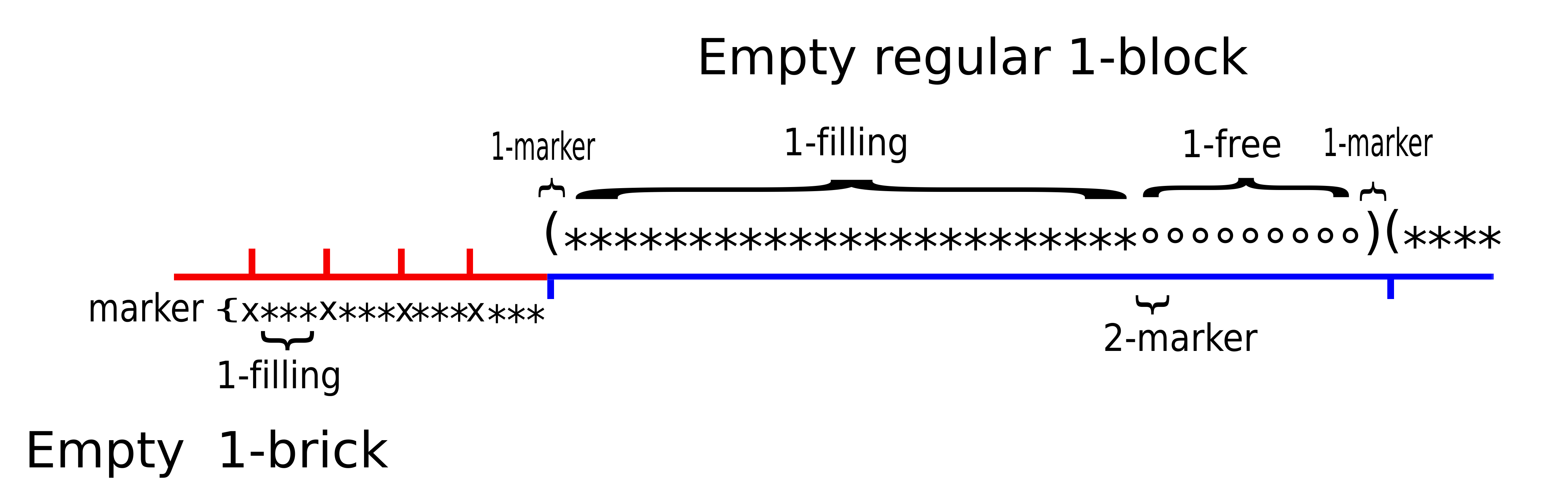}
\captionsetup{width=0.8\textwidth}
\center
\caption{ \textbf{Empty $1$-blocks.} The $2$-marker, $1$-marker,  $1$-filling and $1$-free positions of an empty regular $1$-block and  the marker and $1$-filling positions of  $1$-bricks are represented.}
\end{figure}

\underline{Empty $k$- blocks from empty $k-1$-blocks.}  Empty   $k$-block are obtained from a concatenation of  empty $(k-1)$-blocks  and $l$-bricks with $l\leq k-1$ after the following modifications  which will ensure that all empty regular $k$-blocks  and $k$-bricks (with $k\neq 1$) have a proportion $\alpha/2^{k+1}$ of $k$-free positions.

Let $u^k:=(u_{i_k},...,u_{j_k-1})$ be a concatenation of empty $(k-1)$-blocks  and $l$-bricks with $l\leq k-1$ in such a way  $u^k$ defines a regular $k$-block. Either $i_k$ belongs to $E_{k-1}$, which means $u^k$ starts with a empty regular $(k-1)$-block,   and  then we  change to $[$  the  value at the $k$-marker coordinate  of this $(k-1)$-block or $i_k$ is in $F_{k-1}$, which means $u^k$ starts with a empty $l$-brick with $l\leq k-1$, and then we change to $[$ (or $\times [$ when $l=1$) the value at the marker position of this brick.  We also modify the $k$-marker position of the last $(k-1)$-regular block or $l$-brick of our $k$-block $u^k$ by letting  the symbol  $]$ (or  $\times ]$ when $l=1$).  The marker positions of the remaining $(k-1)$-regular blocks and $l$-bricks with $l\leq k$ in $u^k$ are changed to $+$.

 \begin{figure}[!ht]
\vspace{-0,3cm}
\center
\includegraphics[scale=0.2]{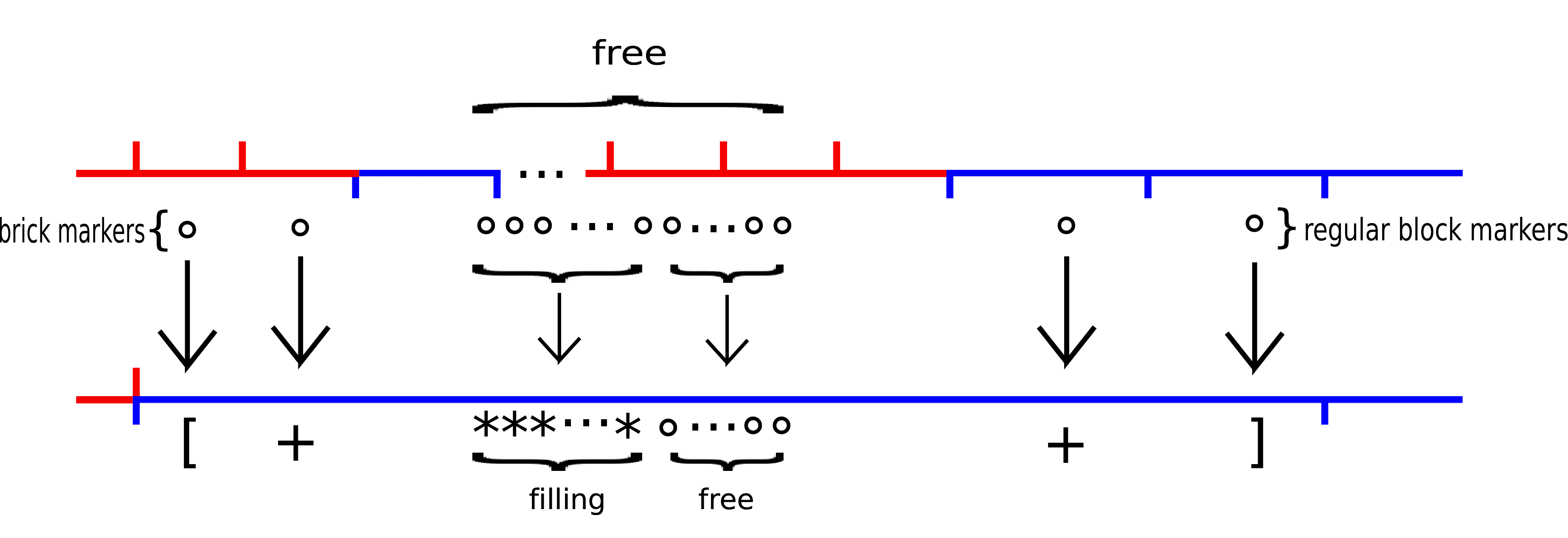}
\caption{\textbf{$(k+1)$-regular block from $k$-blocks.} A regular $(k+1)$-block is obtained from a concatenation of regular $k$-blocks and $l$-bricks with $l\leq  k$.}
\end{figure}

 We finally modify the values in the empty maximal  singular $1$-subblocks $v^k$ (i.e. maximal repetition of a $1$-brick) in $u^k$ as follows. In such a subblock $v^k$, we make $k$-free the  $l^{th}$ positions with $l> \max \left( [|v^k|(1-\alpha/2^k)] , m_1+1\right)$, where $m_1$ is the length of  the repeated $1$-bricks in $v^k$, by letting the value $\circ $ at these coordinates, with one exception : we only let the potential symbol $]$ at the marker position of the last $1$-brick of $v_k$ (it appears when $v^k$ is the end of $u^k$, see Figure \ref{dec}.).\\
 
 \begin{figure}[!h]
 \hspace*{-1,6cm}
\includegraphics[scale=0.22]{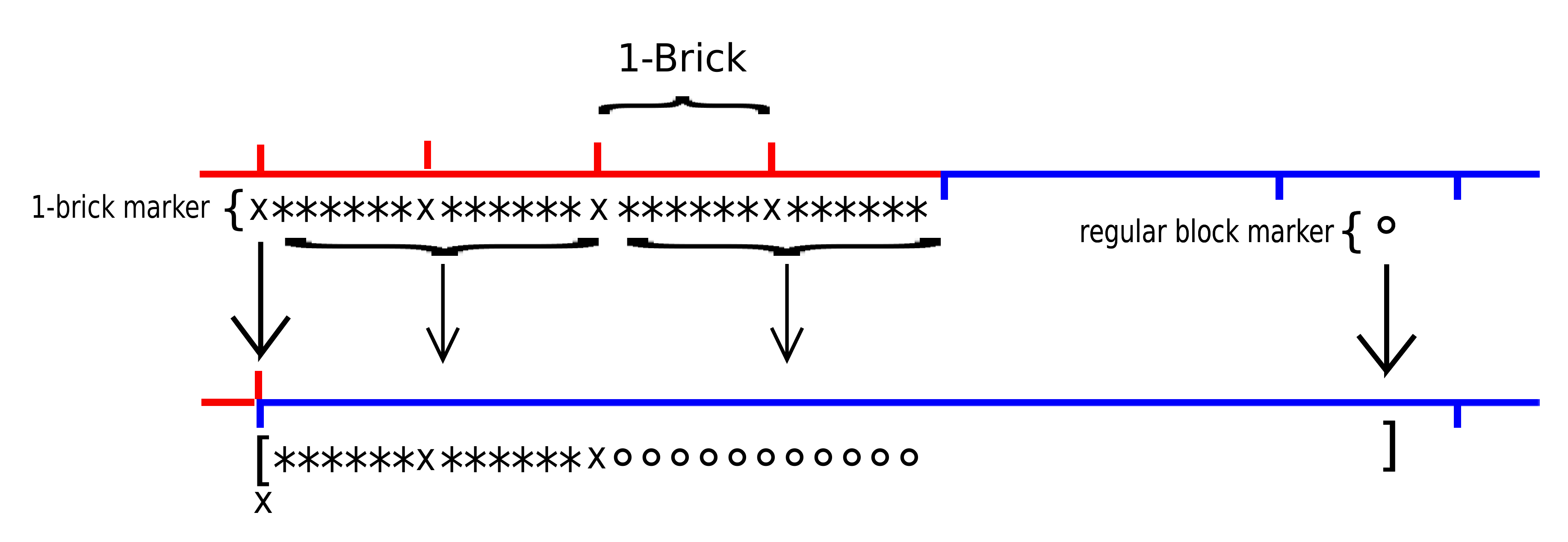} 
\caption {\label{fddf}\textbf{$1$-brick in a regular $(k+1)$-block (or $(k+1)$-brick).} We illustrate the particular case of $1$-bricks in a $k$-singular block which take part in a $(k+1)$-regular block (or a $(k+1)$-brick). Some of the $1$-filling positions of these $1$-bricks are deleted. Such a change occurs only in this situation.}
\end{figure}

Assume now that $u^k:=(u_{i_k},...,u_{j_k-1})$ defines a $l$-brick with $l\leq k$. When $l\leq k-1$ we let the brick remain  unchanged.   Then an empty $k$-brick is a concatenation of  $(k-1)$-regular blocks  or $l$-bricks with $l<k$ (see Figure \ref{gu}). 

 \begin{figure}[!b]
 \vspace*{-1,4cm}
\hspace*{-0,2cm}
\includegraphics[scale=0.26]{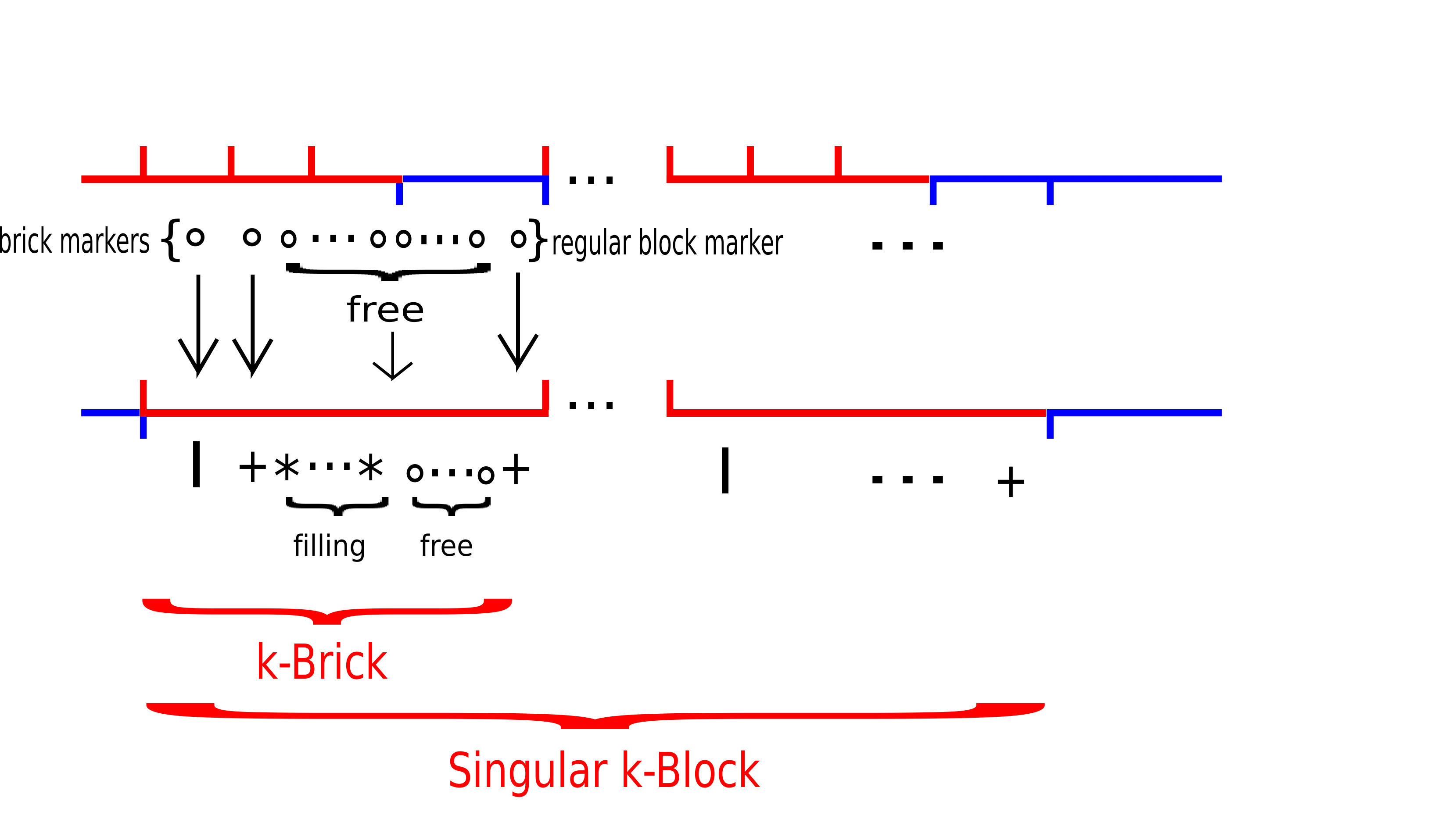}
\caption{\label{gu}\textbf{$k$-brick from $(k-1)$-blocks.} In the first line, we can see the $(k-1)$-blocks. The second line represents the  $k$-singular block given by the repetition of  a $k$-brick, which is a concatenation of a $(k-1)$-regular block and  $l$-bricks with $l< k$. We also visualize the changes to get the empty $k$-code from the empty $(k-1)$-code.}
\end{figure}

To mark the period of this $k$-brick we change to  $|$ (or $\times |$ in the case of $1$-bricks) the symbol at the $k$-marker position of the $(k-1)$-regular block or  at the marker position of the $l$-brick by which the $k$-brick begins. We then make $k$-free the positions in the  maximal singular $1$-subblocks  of $u^k$ in the same way as above.

The  $k$-filling positions (resp. $k$-free) of a $k$-regular block or a $k$-brick $u^k$ are the  $\alpha/2^k|u^k|$ first  (resp. remaining) $(k-1)$-free positions of $u_k$. Finally the $(k+1)$-marker position of a regular $k$-block (resp. the marker position of a $k$-brick) is the first $k$-free position in this block (resp. in this brick).   The empty $k$-code $\psi^\emptyset_k$ with respect to $G$ is then just the infinite word  obtained by concatenation of  all empty $k$-blocks.\\

\textit{2. First scale encoding.}
 Empty regular $1$-blocks are encoded as for aperiodic systems. For any $n\leq n_1$ we consider an injection $\psi_n$ from $Per_n(X,T)$ to $\{1,...,K\}^n$. Now we fill up an empty $1$-brick by writing the image by $\psi_n$ of the associated  distinguished periodic point (without erasing the marker position $\times $).\\
 
 \begin{figure}[!h]
 \vspace{-1.8cm}
  \hspace*{-1,4cm}   
\includegraphics[scale=0.202]{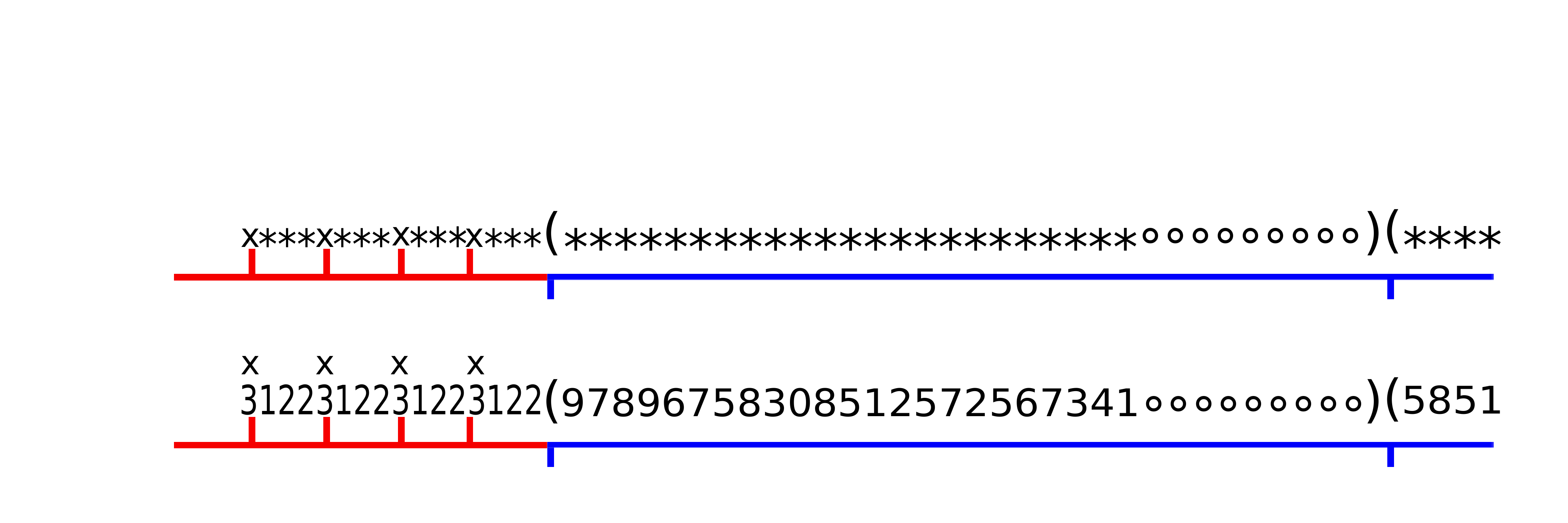}
\caption{\textbf{Filling up empty $1$-blocks ($K=9$).} We fill up empty $1$-blocks. For regular $1$-blocks we proceed as in the aperiodic case. For singular $1$-block we repeat the code of the associated periodic point by preserving the marker $\times$.} 
\end{figure}
 
\textit{3. Higher scale encoding.}
Empty regular $k$-blocks are encoded  by writing in the $k$-filling positions   the code with respect to  $\mathcal{V}_k$ conditionally to $\mathcal{V}_{k-1}$ as in the aperiodic case. We explain now the procedure for an $l$-brick with $l\leq k$. Such a brick is associated to a distinguished periodic point  with least  period $m\in ]n_{l-1},n_l]$. We distinguish then two cases :
either  $l\leq k-1$ and we  let $\psi_k$ be equal to $\psi_{k-1}$ on this brick, or $l=k$ and in the first $\alpha n/2^k$ $k$-filling positions of the brick we write the code of $\mathcal{V}_k^n$ conditionally to $\mathcal{V}_{k-1}^n$ and in the next  $\alpha n/2^k$ $k$-filling positions we specify the periodic point by asymptotic $per$-expansiveness according to  Inequality (\ref{peexp}) on p.\pageref{peexp}.

 The sequence $(\psi_k(x))_k$ is converging pointwisely for the product topology in $\mathcal{A}^\mathbb{Z}$ (with $\mathcal{A}$ being the finite alphabet given by all symbols previously used) as  for any $i$ the $i^{th}$ coordinate is changed at most two times. Indeed,  $k$-free positions are changed at most one time.  For $k\geq 2 $ the  $k$-filling positions are not changed  but for $k=1$ some of them become free and then filled again.  However this  last case may happen at any coordinate at most one time as we leave $1$-bricks forever.  We let $\psi(x)$ be the pointwise limit of $\psi_k(x)$. \\

\textit{4. Decoding.} The  $k$-code $\psi_k(x)$ may be deduced  on $k$-regular blocks from 
$\psi(x)$ by identifying $k$-blocks by induction as follows. Firstly regular $1$-block correspond to finite words in $\psi(x)$ delimited by two consecutive parentheses $($ and $)$. In these blocks we may read the associated element of $\mathcal{V}_1^n$ by looking at the filling $1$-positions. Now in a $1$-singular  block a complete $1$-bricks has not been distorted by higher scale encoding (see Figure \ref{fddf}) and we may read the associated  periodic point between the symbols $\times$.

Then regular $2$-blocks are identified as words between consecutive marked times by  $[$ and $]$ with length in $[n_2/2, 4n_2]$, where these markers occur at the $2$-marker position of a $1$-regular block or at a marker position of a $1$-brick. Note that we may find in a $2$-singular block some consecutive parenthesis $[$ and $]$  coming from the $k$-encoding with $k>2$, however they are at least $n_3/2$-far from each other.

The length of $2$-bricks  may be found thanks to the symbols $|$ at the $2$-marker position of $1$-block or at  a marker position of a $1$-block and we may identify the associated periodic point at the $2$-filling positions.

Similarly one may recover $\psi_k(x)$ and the $k$-blocks from $\psi(x)$  by induction on $k$ and then read either the associated periodic point $\epsilon_k$-close to the piece of orbit for singular $k$-block or the  orbit with respect to $\mathcal{V}_k$ for regular ones. We can assume $2\epsilon_k$ is less than the Lebesgue number of the open cover $T^l\mathcal{V}_k$ for any $k$ and any $|l|\leq n_k$.  In particular this proves that any two points with the same image under $\psi_k$ belong to the same element of $\mathcal{V}_k$. Thus $\psi_k$ is $\epsilon'_k$-injective with $\epsilon'_k$ being the diameter of $\mathcal{V}_k$  and $\psi$ is injective. 

\begin{figure}[!h]
\hspace*{-2,1cm}
\includegraphics[scale=0.127]{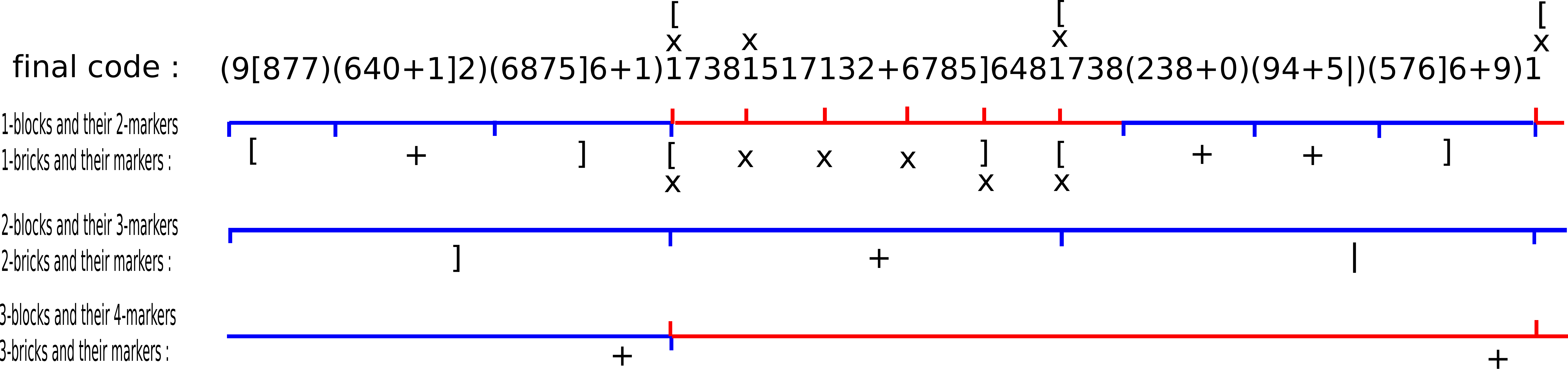}
\caption{\label{dec}\textbf{Example of decoding.} We recover from a part of a  limit code $\psi(x)$ the associated structure of $k$-blocks (resp. bricks) and their $(k+1)-$markers (resp. markers) for $k=1,2,3$.}
\end{figure}

\textit{5. Continuity and convergence  of $(\psi_k)_k$.} One proves as in the aperiodic case the codes $(\psi_k)_k$ are continuous by the clopen property of $(\mathcal{V}_k)_k$, $(U_k)_k$ and $(Per_{n_k}^{\epsilon_k})_k$.
As we change at each step a proportion $\alpha/2^k$ of coordinates  on subblocks of size larger than $n_k/2$ and less than or equal to  $4n_k$ the sequence $(\psi_k)_k$ is converging again uniformly  with respect to $(d_N)_N$.  \\

\textit{6.  Principal strongly faithful symbolic extension as the inverse of Krieger embedding.}
The inverse of $\psi$ is (uniformly) continuous on $\psi(X)$. Indeed if $\psi(x)$ and $\psi(y)$ coincide on their $[-8n_k,8n_k]$  coordinates then  the points $x$ and $y$ belong  either to the same regular  $k$-block or  are $\epsilon_k$-close to the same $n_k$-periodic point. In both cases  $x$ and $y$ belong to the same element of $\mathcal{V}_k$. The other desired properties of the extension are shown as in the aperiodic case.

\section{The general case}\label{red}

Now we will deduce  the general case from the zero-dimensional one by using the small boundary property. More precisely this last property allows us to embed the system in a zero-dimensional one as follows :

\subsection{Embedding systems with the small boundary property in zero-dimensional ones}

\begin{prop}\label{dddd}(p. 73 \cite{Dowstru})
Let $(X,T)$ be an (resp. aperiodic, resp. asymptotically $per$-expansive, resp. asymptotically $h$-expansive)  topological dynamical system with the  small boundary property. Then there exists a zero dimension (resp. aperiodic , resp. asymptotically $per$-expansive,        resp. asymptotically $h$-expansive) system $(Y,S)$ and a Borel  equivariant injective map $\psi:X\rightarrow Y$ such that
\begin{itemize}
\item $\psi$ is continuous on a full set,
\item the induced map $\psi^*:\mathcal{M}(X,T)\rightarrow \mathcal{M}(Y,S)$ is an homeomorphism,
\item $\psi^{-1}$ is uniformly continuous on $\psi(X)$ and  extends to a continuous principal  strongly faithful extension $\pi$ on $Y=\overline{\psi(X)}$  of $(X,T)$.
\end{itemize}
\end{prop}

For completeness we sketch the proof of Proposition \ref{dddd}. We consider a topological dynamical system with the small boundary property.  
Let $(P_k)_k$ be a nonincreasing sequence of essential partitions of $X$ whose diameter goes to zero.   We consider the shift $\sigma$ over $\prod_k P_k^\mathbb{Z}$. Let $\psi:X\rightarrow \prod_k P_k^\mathbb{Z}$ be defined by $\psi(x)=(P_k(T^nx))$, it is injective and  semi-conjugates $T$ with the shift, $\sigma \circ \psi=\psi\circ T$. Clearly $\psi$ is continuous on the full subset of points whose orbit does not intersect the boundary of the $P_k$'s. Moreover  the induced map $\psi^*$ is a topological embedding. Indeed to prove the continuity of $\psi^*$ it is enough to see that $\psi^*\nu(A)$ converges to $\psi^*\mu(A)$ when $\nu$ goes to $\mu$ for any finite cylinder $A$ in $ \prod_k P_k^\mathbb{Z}$. But for such $A$ the closure and the interior of $\psi^{-1}A$ have the same measure for any invariant measure so that the previous property of continuity holds.

 The inverse $\psi^{-1}$ satisfies $\psi^{-1}\left((A_n^k)_{n,k}\right)=\bigcap_{n,k}\overline{T^{-n}A_{n,k}}$ and is thus clearly uniformly continuous. The continuous extension $\pi$ of $\psi^{-1}$ maps $\overline{\psi(X)}\setminus \psi(X)$ on the boundaries of the $(T^nP_k)_{n,k}$'s and thus any  $\sigma$-invariant measure on $\overline{\psi(X)}$ is supported on $\psi(X)$. As  $\psi(Per_n(X,T))=Per_n(\psi(X),S)$ for any integer $n$ we conclude that  $\pi$  is a strongly  faithful principal extension of $(X,T)$. 
 
 Principal extensions preserves the asymptotic $h$-expansiveness property and strongly faithful  extensions preserves the aperiodicity. Finally if  $n$-periodic points in $Y=\overline{\psi(X)}$ lie in the same $P_k^n$ for some $k$ their $\pi$-image  are $n$-periodic points for $T$ in the same $\overline{P_k}^n$. As these periodic points are in fact in $\psi(X)$ and $\pi$ is one-to-one on $\psi(X)$, we conclude asymptotic  $per$-expansiveness is also preserved by this extension.  
 
\subsection{Reduction to the zero-dimensional case}
By applying the above Proposition \ref{dddd}, our Main Theorem follows from Proposition \ref{zer}. Indeed if $(X,T)$ is as in our Main Theorem, then  the zero-dimensional extension $(Y,S) $ of $(X,T)$  obtained in the above proposition satisfies the assumption of Proposition \ref{zer}. The desired embedding is then given by the composition of the embedding in the zero-dimensional system $(Y,S)$ and the Krieger embedding of $(Y,S)$ in the shift with the desired finite alphabet. The property of convergence and continuity of the embedding and its inverse follows then from basic properties of composition.

\section{Proof of Propostion \ref{embb} and \ref{smbb}}

\subsection{Topological embedding in the set of invariant measures  $\Rightarrow$ asymptotic $h$-expansiveness}
Assume $\psi:X\rightarrow \Lambda^\mathbb{Z}$ is a Borel equivariant embedding such that   $\psi^*$ is a topological embedding. Recall that $P_0$ denotes  the zero coordinate partition of $\Lambda^\mathbb{Z}$. Then for any $\mu\in \mathcal{M}(X,T)$ we have  $h(\mu)=h(\psi^*\mu)=h(\psi^*\mu,P_0)=h(\mu,\psi^{-1}P_0)$. In the present paper such a partition $\psi^{-1}(P_0)$ is said  to be generating.  Let $B\in P_0$ and $A=\psi^{-1}B \in \psi^{-1}P_0$ and let $\mu$ be  $T$-invariant measure.  As $\psi^*$ is continuous and $B$ is a clopen set  we have 
\begin{eqnarray*}
\limsup_{\nu\rightarrow \mu}\nu(A)&=&\limsup_{\nu\rightarrow \mu}\psi^*\nu(B),\\
&\leq&\lim_{\xi \rightarrow \psi^*\mu} \xi(B),\\
&\leq &\psi^*\mu(B)=\mu(A).
\end{eqnarray*}
Similarly we prove $\liminf_{\nu\rightarrow \mu}\nu(A)\geq \mu(A)$. Therefore  for any $A\in\psi^{-1}P_0$, the function $\mu\mapsto \mu(A)$ is continuous on $\mathcal{M}(X,T)$. In general it does not imply that $A$ has  a small boundary, for example when there is  an attracting fixed point in the boundary of $B\in \psi^{-1}P_0$ (the Dirac measure at this point is then isolated in $\mathcal{M}(X,T)$).  Finally asymptotic $h$-expansiveness follows from the following lemma. 

\begin{lemma}
Let $(X,T)$ be a topological dynamical system admitting a finite generating partition $P$ such that for any $A\in P$, the function $\mu\mapsto \mu(A)$ is continuous on $\mathcal{M}(X,T)$, then $(X,T)$ is asymptotically $h$-expansive. 
\end{lemma}

\begin{proof}
Let $(h_k)_k$ be the Lebegue  entropy structure (see Section 6.3 in \cite{Dowstru}). For any integer $k$ the function $h_k$ is defined as the entropy of $\mu\times \lambda$ with respect to  an essential partition $P_k$ of $(X\times \mathbb{S}^1,T\times R)$, where $R$ is an irrational circle rotation and $(P_k)_k$ is a nonincreasing sequence of essential partitions of $X\times \mathbb{S}^1$ whose diameters go to zero. By the tail variational principle (Theorem 5.1.4 in \cite{Dowstru}) $T$ is asymptotically $h$-expansive if 
and only if $(h_k)_k$ is converging uniformly to $h$. By assumption on $P$, the maps $\mu\mapsto  h(\mu\times \lambda , P \times \mathbb{S}^1 \vee P_k | P_k)$ are upper semicontinuous on $\mathcal{M}(X,T)$. Moreover they are   nonincreasing in $k$ and  converging to zero. It is well known that the convergence is in fact uniform for such sequences of functions. As  $h(\mu)-h_k(\mu)\leq  h(\mu\times \lambda , P \times \mathbb{S}^1 \vee P_k | P_k)$ for any $T$-invariant measure $\mu$ the sequence $h-h_k$  is  also converging uniformly to zero and thus $T$ is asymptotically $h$- expansive.
\end{proof}

\subsection{Topological embedding in the set of invariant measures  $\Rightarrow$ asymptotic $per$-expansiveness}
Let $\psi:(X,T)\rightarrow (\{1,...,K\}^\mathbb{Z},\sigma)$ be a Borel equivariant embedding such that   $\psi^*$ is a topological embedding and let $(\mathcal{U}_k)_k$ be a  sequence of open covers such that the diameter of $\mathcal{U}_k$ goes to zero.
We will show that the exponential growth in $n$ of the set $U_k^n\cap Per_{n}(X,T)$ goes to zero uniformly in $U_k^n\in \mathcal{U}_k^n$ when $k$ goes to infinity. Assume first $(X,T)$ has the small boundary property. Let $(P_l)_l$ be a sequence of essential partitions with $diam(P_l)\xrightarrow{l\rightarrow 0}0$. Let $\epsilon>0$. For any $l$ there exist 
$k_l$ and $n_l$ such that any $U_{k}^n$ with $k>k_l$ and $n>n_l$ meets at most $ e^{\epsilon n}$ elements of $P_{l}^n$ (see Lemma 6 in \cite{burg}). Thus it is enough to prove that the exponential growth in $n$ of the set $A^n\cap Per_{n}(X,T)$ goes to zero uniformly in $A^n\in P_l^n$ when $l$ goes to infinity. Take $A^n$ so that the previous intersection has maximal cardinality. We consider 
 the associated empirical measures $\mu_n:=\frac{1}{\sharp A^n\cap Per_{n}(X,T) } \sum_{x\in A^n\cap Per_{n}(X,T)}\delta_x$ and  $\nu_n:=\frac{1}{n}\sum_{k=0}^{n-1}T^*\mu_n$ for any integer $n$. Let $P=\psi^{-1}(P_0)$. Following \cite{burg} p.368 l.6,  we get for any $m<n$ 
\begin{eqnarray*} 
 \frac{H _{\nu_n}(P^m|P_l^m)}{m}&=&\frac{H _{\psi^*\nu_n}(P_0^m | \psi(P_l)^m)}{m};\\
 &\geq & \frac{H _{\psi^*\mu_n}(P_0^n | \psi(P_l)^n)-3m\log K}{n};\\
& \geq &\frac{\log \sharp A^n\cap Per_{n}(X,T)-3m\log K}{n}.
\end{eqnarray*}
The last inequality follows from the injectivity of $\psi$, the inclusion  $\psi(Per_n(X,T))\subset Per_n(\Lambda^\mathbb{Z},S)$ and  the inequalities  $\sharp Per_n(\Lambda^\mathbb{Z},S)\cap B^n\leq 1$ for any given $B^n\in P_0^n$.  Observe also that $\nu_n$ is invariant: it is the barycenter of the periodic measures associated to the periodic points in  $A^n\cap Per_{n}(X,T)$.  Let $\nu$ be a weak limit of $(\nu_n)_n$. 
As seen above in Subsection 6.1 the sequence $(\nu_n(A))_n$ converges to $\nu(A)$ for any $A\in P$.
By taking the limit in $n$ and then $m$ we get thus by upper semicontinuity:
$$h(\nu, P|P_l)\geq \limsup_n \frac{\log \sharp A^n\cap Per_{n}(X,T)}{n}.$$
However by asymptotic $h$-expansiveness, which follows from Subsection 6.1, the left member goes to zero uniformly in $\nu$ when $l$ goes to infinity. This concludes the proof when $(X,T)$ has the small boundary property. For general systems one may consider the product with an irrational circle rotation and apply the previous proof by replacing $\mu_n$ with is product with a Dirac measure of the circle, $\psi$ by its product with the circle identity and $P$ by its product with the circle.

\subsection{Principal strongly faithful asymptotic  $per$-expansive extension  $\Rightarrow$ asymptotic $per$-expansiveness}\label{dd}We prove now the second item of Proposition \ref{embb}. We only have to prove asymptotic $per$-expansiveness as principal extensions preserve asymptotical $h$-expansiveness as proved by Ledrappier \cite{Led}.
Let $\pi:(Y,S)\rightarrow (X,T)$ be a principal  extension. Equivalently the topological conditional entropy  $h(S|T)$ of $S$ w.r.t. $T$ vanishes (see Section 6.3 of \cite{Dowb} for precise definitions and basic properties). In particular, for any $\epsilon>0$ there exists an open cover $\mathcal{U}$ of $X$ such that for any open cover $\mathcal{V}$ and for any $U^n\in \mathcal{U}^n$ with large $n$, the set $\pi^{-1}U^n$ may be covered by $e^{\epsilon n/2}$ element of  $\mathcal{V}^n$. If we assume moreover the extension to be strongly faithful it preserves periodic orbits. In particular the number of $n$-periodic points for $T$ in $U^n$
is equal to the number of $n$-periodic points for $S$ in $\pi^{-1}U^n$. Therefore if $(Y,S)$ is itself asymptotically $per$-expansive, one may choose $\mathcal{V}$ so that  the number $n$-periodic points for $S$ in any $V^n$ is less than $e^{\epsilon n/2}$. One concludes that there at most $e^{\epsilon n}$ $n$-periodic points of $(X,T)$ in  $U^n$. This proves $(X,T)$ is asymptotically $per$-expansive.

\subsection{$\epsilon$-injective essentially continuous map $\Rightarrow$ small boundary property}

In Krieger Theorem for topological dynamical systems $(X,T)$ recalled in Theorem \ref{kkkk}, the existence of a topological embedding from $X$ to a shift space with finite alphabet  forces the space $X$ to be zero-dimensional. Analogously we have now : 

\begin{prop}\label{smbb} Let $(X,T)$ be a topological dynamical system and let $Y$ be a zero-dimensional topological space. Assume that for any $\epsilon>0$ there exists an essentially continuous $\epsilon$-injective map from $X$ to $Y$. Then $(X,T)$ has the small boundary property. \end{prop}

\begin{proof}
Let $(X,T)$ be a topological dynamical system and let $Y$ be a zero-dimensional space.
Assume that for any $\epsilon>0$ there exists an $\epsilon$-injective essentially continuous map $\psi$. Let $\delta>0$ be such that for 
any set $Y$ with  diameter less than $\delta$ the set $\psi^{-1}Y$ has diameter less than $\epsilon$. We let $\mathcal{B}$ be a basis of clopen sets such that 
$\psi^{-1}B$ has a  small boundary for any $B\in \mathcal{B}$.  Consider then a finite clopen partition $P$ of $Y$ with diameter less than $\delta$ such that any element of $P$ is a finite intersection of elements of $\mathcal{B}$. In particular, for any $A\in P$, the set $\psi^{-1}A$ is a finite intersection of small boundary sets and thus it has itself small boundary. As the diameter of $\psi^{-1}A$ is less than $\epsilon$, it concludes the proof of Proposition \ref{smbb}.
\end{proof}

\appendix
\section{Appendix : Equality in Krieger topological embedding problem}

\begin{prop}\nonumber
Let $(X,T)$ be a topological dynamical system with $h_{top}(T)=\log K$. Assume also there is a topological equivariant embedding $\psi:X\rightarrow \{1,...,K\}^\mathbb{Z}$. Then $\psi$ is a topological conjugacy. 
\end{prop}

\begin{proof}
It is enough to prove that $\psi$ is onto. Let $P$ be the clopen cover of $X$ given by $P:=\psi^{-1}P_0$. Then for $\delta>0$ and $n\in \mathbb{N}$, any  $(\delta,n)$ separated set $E$ in $A^n\in P^n$ has cardinality bounded by a constant. Indeed there  is $\delta'>0$ depending only on $\delta$ and $\psi$ such that $\psi(E)$ is $(\delta',n)$ separated in $\psi(A^n)\in P_0^n$. In particular $h_{top}(T)=h_{top}(T, P):=\inf_n\frac{\log\sharp P^n}{n}$. But as $\sharp P=K$ and $h_{top}(T)=\log K$ we have $\sharp P^n=K^n$ for all $n\in \mathbb{N}$. Fix $y\in \{1,...,K\}^\mathbb{Z}$ and let us show there exists $x\in X $ with $\psi(x)=y$. Let $n\in \mathbb{N}$.
As $\sharp P^{2n+1}=K^{2n+1}$ there is $x_n\in X$ such that the $k^{th}$ coordinates of $\psi(x_n)$ with $|k|\leq n $ coincide with those of $y$. Then if $x$ is an accumulation point of $(x_n)_n$  we get $\psi(x)=y$ by continuity of $\psi$. 
\end{proof}

\end{document}